\documentclass[12pt]{amsart}
\usepackage{amscd,setspace,amssymb,amsopn,amsmath,amsthm,mathrsfs,lmodern,graphics,amsfonts,enumerate,verbatim,calc
} 
\usepackage[all]{xy}
\usepackage[colorlinks=true, citecolor=PineGreen, filecolor=PineGreen, linkcolor=Purple,pagebackref, hyperindex]{hyperref}
\usepackage[dvipsnames]{xcolor}
\usepackage{todonotes}
\usepackage{tikz-cd}
\usepackage{color, hyperref}
\usepackage[OT2,OT1]{fontenc}

\newcommand\cyr{%
\renewcommand\rmdefault{wncyr}%
\renewcommand\sfdefault{wncyss}%
\renewcommand\encodingdefault{OT2}%
\normalfont
\selectfont}
\DeclareTextFontCommand{\textcyr}{\cyr}
\usepackage{amssymb,amsmath}
\DeclareFontFamily{OT1}{rsfs}{}	
\DeclareFontShape{OT1}{rsfs}{n}{it}{<-> rsfs10}{}
\DeclareMathAlphabet{\fmathscr}{OT1}{rsfs}{n}{it}
\numberwithin{equation}{section}
\newtheorem{Theoremx}{Theorem}
 
\newtheorem{theorem}{Theorem}[section]
\newtheorem{lemma}[theorem]{Lemma}

\newtheorem{proposition}[theorem]{Proposition}
\newtheorem{corollary}[theorem]{Corollary}
\newtheorem{claim}[theorem]{Claim}
\newtheorem{question}{Question}

\theoremstyle{definition}
\newtheorem{definition}[theorem]{Definition}

\theoremstyle{remark}

\newcommand{\Ass}{\operatorname{Ass}}
\newcommand{\im}{\operatorname{Im}}
\renewcommand{\ker}{\operatorname{Ker}}

\newcommand{\Spec}{\operatorname{Spec}}

\newcommand{\Ht}{\operatorname{ht}}
\newcommand{\Height}{\operatorname{ht}}
\newcommand{\di}{\operatorname{div}}

\newcommand{\gr}{\operatorname{gr}}

\newcommand{\Mod}{\operatorname{Mod}}

\newcommand{\Ext}{\operatorname{Ext}}

\newcommand{\Hom}{\operatorname{Hom}}

\newcommand{\Ann}{\operatorname{Ann}}

\newcommand{\depth}{\operatorname{depth}}

\newcommand{\coker}{\operatorname{Coker}}

\newcommand{\Proj}{\operatorname{Proj}}

\newcommand{\e}{\operatorname{e_{HK}}}

\newcommand{\rank}{\operatorname{rank}}
\newcommand{\frk}{\operatorname{frk}}
\newcommand{\lcb}{\operatorname{lcb}}

\newcommand{\N}{\mathbb{N}}

\newcommand{\Q}{\mathbb{Q}}

\newcommand{\fm}{\mathfrak{m}}
\newcommand{\fp}{\mathfrak{p}}

\newcommand{\fa}{\mathfrak{a}}
\newcommand{\fb}{\mathfrak{b}}

\newcommand{\fn}{\mathfrak{n}}

\begin{document}
\title{Local cohomology bounds and test ideals}

\author[Ian Aberbach]{Ian Aberbach}
\address{Department of Mathematics, University of Missouri, Columbia, MO 65211, USA}
\email{aberbachi@missouri.edu}

\author[Thomas Polstra]{Thomas Polstra}
\thanks{Polstra was supported in part by NSF Postdoctoral Research Fellowship DMS $\#1703856$.}
\address{Department of Mathematics, University of Virginia, Charlottesville, VA 22903 USA}
\email{tp2tt@virginia.edu}


\begin{abstract} We find sufficient conditions which imply equality of the finitistic test ideal and test ideal in rings of prime characteristic. Utilizing recent progress from the minimal model program we equate the notions of $F$-regular and strongly $F$-regular for 4-dimensional rings essentially of finite type over a field of prime characteristic $p>5$.
\end{abstract}

 \maketitle


\section{Introduction}

Introduced and developed by Hochster and Huneke in \cite{HHStrongFregular, HHJAMS, HHsmall, HHPhantom, HHTAMS, HHJAG}, the theory of tight closure is a central topic in the study of Noetherian rings of prime characteristic $p>0$. Suppose $R$ is a Noetherian ring of prime characteristic $p>0$ and let $R^\circ$ be the set of elements which avoid all minimal primes of $R$. Let $I\subseteq R$ be an ideal of $R$ and denote by $I^{[p^e]}$ the expansion of $I$ along the $e$th iterate of the Frobenius endomorphism. The tight closure of $I$ is the ideal $I^*$ consisting of elements $x\in R$ such that there exists an element $c\in R^{\circ}$ with the property that $cx^{p^e}\in I^{[p^e]}$ for all $e\gg 0$. A defining problem of tight closure theory was the question of whether or not tight closure commutes with localization: If $W$ is a multiplicative set and $I\subseteq R$ an ideal, is $I^*R_W=(IR_W)^*$? There are scenarios when tight closure does commute with localization, e.g., \cite{AHH} and \cite{Yaofinite}. However, there exist hypersurfaces for which tight closure does not commute with localization, \cite{BrennerMonsky}.  Brenner's and Monsky's counterexample to the localization problem leaves open the intriguing problem if the property of tight closure being a trivial operation on ideals commutes with localization.

Continue to let $R$ be a Noetherian ring of prime characteristic $p>0$. The ring $R$ is called \emph{weakly $F$-regular} if every ideal is tight closed, that is $I=I^*$ for every ideal $I$.\footnote{A defining property of tight closure theory is that every regular ring is weakly $F$-regular.} A ring is called \emph{$F$-regular} if every localization of $R$ is weakly $F$-regular. Let $F^e_*R$ denote the restriction of scalars of $R$ along the $e$th iterate Frobenius endomorphism $F^e:R\to R$. We say that $R$ is \emph{strongly $F$-regular} if for each nonzero element $c\in R$ there exists $e\in \N$ such that the $R$-linear map $R\to  F^e_*R$ defined by $1\mapsto F^e_*c$ is pure. Every strongly $F$-regular ring is weakly $F$-regular and the property of being strongly $F$-regular passes to localization. It is conjectured that all three notions of $F$-regularity agree. Utilizing recent progress from the prime characteristic minimal model program, \cite{DasWaldronArxiv}, our main contribution towards this problem is the following:

\begin{Theoremx}\label{Main theorem dimension 4} Let $(R,\fm,k)$ be a ring of Krull dimension no more than $4$ which is essentially of finite type over a field of prime characteristic $p>5$. If $R$ is $F$-regular then $R$ is strongly $F$-regular.
\end{Theoremx}

A key ingredient to the proof of Theorem~\ref{Main theorem dimension 4} is that $3$-dimensional strongly $F$-regular rings essentially of finite type over a field of prime characteristic $p>5$ have Noetherian anti-canonical algebra. More specifically, Das and Waldron have proven every $3$-dimensional local KLT singularity pair essentially of finite type over a field of prime characteristic $p>5$ has Noetherian anti-canonical algebra, \cite[Corollary~1.12]{DasWaldronArxiv}. If $(R,\fm,k)$ is a local strongly $F$-regular ring then there exists an effective boundary divisor $\Delta$ such that $(\Spec(R),\Delta)$ is globally $F$-regular (or just $F$-regular since $\Spec(R)$ is affine), \cite[Corollary~6.9]{SchwedeSmith}, and therefore has KLT singularities by \cite[Theorem~3.3]{HaraWatanabe}. Hence every $4$-dimensional $F$-regular ring essentially of finite type over a field of prime characteristic $p>5$ satisfy the hypotheses of Theorem~\ref{Main theorem dimension 4} and therefore are strongly $F$-regular. Moreover, every weakly $F$-regular ring finite type over a field with infinite trancendence degree over $\mathbb{F}_p$ with $p>5$ and of Krull dimension no more than $4$ is now known to be strongly $F$-regular, see \cite[Theorem~8.1]{HHTAMS}.

There has been tremendous effort to equate the various notions of $F$-regularity since the theory of tight closure was introduced. Efforts to equate at least two of the notions of $F$-regularity have typically required making desirable geometric assumptions on the ring $R$. For example:

\begin{enumerate}
\item Hochster and Huneke showed weak implies strong for Gorenstein rings, \cite{HHStrongFregular};
\item Weakly $F$-regular is equivalent to strongly $F$-regular whenever $R$ is $\N$-graded over a field by work on Lyubeznik and Smith, \cite{LyubeznikSmith};
\item MacCrimmon showed weakly $F$-regular is equivalent to strongly $F$-regular if $R$ is assumed to be $\mathbb{Q}$-Gorenstein at non-maximal points of $\Spec(R)$, \cite{Maccrimmon};
\item Murthy proved that weakly $F$-regular and $F$-regular are equivalent conditions whenever the ring $R$ is finite type over an uncountable field, see \cite[Theorem~12.2]{HunekeBook}. Hochster and Huneke extended Murthy's result in \cite[Theorem~8.1]{HHTAMS} to rings finite type over a field which has infinite transcendence degree over its prime field $\mathbb{F}_p$;
\item The conjectured equivalence of weak and strong follows by an unpublished result of Singh, provided the symbolic Rees ring associated with an anti-canonical ideal is Noetherian, see \cite[Corollary~5.9]{CEMS}. See also \cite{Aberbach2002} for a related assumption on the anti-canonical ideal from which validity of the weak implies strong conjecture can be derived. 
\end{enumerate}

There has been limited progress on equating the various notions of $F$-regularity without making conjecturely unnecessary assumptions. Williams' theorem, see \cite{Williams}, equates the notions of weakly $F$-regular and strongly $F$-regular for rings of dimension no more than $3$. Williams proof of the weak implies strong conjecture in dimension $3$ relies on Lipman's theorem that the divisor class groups of the local rings of an excellent surface with at worst rational singularities is finite, \cite{LipmanRationalSingularities}. Specifically, Williams uses that the canonical class of a local three dimensional weakly $F$-regular ring is torsion on the punctured spectrum, an assumption MacCrimmon imposed on larger dimensional rings in \cite{Maccrimmon} in order to extend William's methodology to a large class of rings of arbitrarily large dimension. Lipman's theorem on divisor class groups requires an understanding of minimal resolutions of rational surface singularities by quadratic transforms.

In the spirit of MacCrimmon's theorem, we do not limit ourselves to low dimensions. We instead impose desirable geometric conditions on higher dimensional rings to find a new and interesting class of rings for which the finitistic test ideal and the test ideal agree. To motivate our most general result we rephrase MacCrimmon's theorem from \cite{Maccrimmon} in terms of analytic spread of an anti-canonical ideal. Let $(R,\fm,k)$ be a normal Cohen-Macaulay domain of prime characteristic $p>0$ and $\omega_R\cong J_1\subsetneq R$ a canonical ideal. The ideal $J_1$ is of pure height $1$ and so there exists an element $a\in R$ such that $a$ generates $J_1$ at its components. We can then write $(a)=J_1\cap K_1$ where $K_1$ is pure height one and the components of $K_1$ are disjoint from the components of $J_1$. The ideal $K_1$ is an \emph{anti-canonical ideal} of $R$ and is the inverse of $J_1$ as an element of the divisor class group of $R$. To assume $R$ is $\Q$-Gorenstein on the punctured spectrum is equivalent to assuming that for some natural number $N\geq 1$ the the $N$th symbolic power of $K_1$,  $K_1^{(N)}$, has analytic spread $1$ on the punctured spectrum.

We recover MacCrimmon's result by proving every weakly $F$-regular ring is strongly $F$-regular under the milder hypothesis that some symbolic power of an anti-canonical ideal has analytic spread at most $2$ on the punctured spectrum.


\begin{Theoremx}\label{Main tight closure theorem} Let $(R,\fm,k)$ be a local Cohen-Macaulay normal domain of prime characteristic $p$, Krull dimension $d$, and $\Q$-Gorenstein in codimension $2$. Suppose further that $R$ has a canonical ideal and that some symbolic power of the corresponding anti-canonical ideal has analytic spread at most $2$ on the punctured spectrum. Let $E_R(k)$ be the injective hull of the residue field. Then the finitistic tight closure and the tight closure of the zero submodule of $E_R(k)$ agree. In particular, the finitistic test ideal and the test ideal of $R$ agree. Therefore if $R$ is weakly $F$-regular then $R$ is strongly $F$-regular.
\end{Theoremx}

Techniques introduced in this article allow us to equate the finitistic tight closure and the tight closure of the zero submodule of the injective hull of the residue field through a careful analysis of the maps of Koszul cohomologies defining certain local cohomology modules. Our analysis of local cohomology is centered around the notion of a \emph{local cohomology bound} defined in Section~\ref{lcb introduction}.

 The relationship between $F$-signature and relative Hilbert-Kunz multiplicity is also explored.  See \cite{HunekeHK, PolstraTucker} for introductions to the theory of prime characteristic invariants, such as Hilbert-Kunz multiplicity and $F$-signature, in local rings. The $F$-signature of an $F$-finite ring\footnote{The ring $R$ is $F$-finite if $F^e_*R$ is a finitely generated module for some , equivalently each, $e\in \mathbb{N}$.} is the limit
\[
\lim_{e\to \infty}\frac{\frk(F^e_*R)}{\rank(F^e_*R)}
\] 
where $\frk(F^e_*R)$ is the largest rank of an $R$-free summand appearing in a direct sum decomposition of $F^e_*R$. The invariant $F$-signature was shown to exist under the local hypothesis in \cite{Tucker}. If $(R,\fm,k)$ is local of Krull dimension $d$ then the Hilbert-Kunz multiplicity of an $\fm$-primary ideal $I$ is the limit
\[
\e(I)=\lim_{e\to \infty}\frac{\lambda(R/I^{[p^e]})}{p^{ed}}
\]
where $\lambda(R/I^{[p^e]})$ denotes the length of $R/I^{[p^e]}$. Existence of Hilbert-Kunz multiplicity was established by Monsky, \cite{Monsky}.

 If $(R,\fm,k)$ is local, $I$ an $\fm$-primary ideal, and $u\not\in I$ then 
\[
s(R)\leq \e(I)-\e((I,u)),
\]
\cite[Proposition~15]{HunekeLeuschke}. Watanabe and Yoshida explored the notion of minimal realtive Hilbert-Kunz multiplicity and its relation with $F$-signature in \cite{WY}. They suspected that the $F$-signature of $R$ is realized as the minimum of all relative Hilbert-Kunz multiplicities. For example, if $(R,\fm,k)$ is Gorenstein, $I$ a parameter ideal, and $u\in R$ generates the socle mod $I$, then $s(R)=\e(I)-\e(I,u)$ by \cite[Theorem~11]{HunekeLeuschke}. More generally, it is known that the $F$-signature of a local ring agrees with the infimum of all relative Hilbert-Kunz multiplicities by work of the second author and Tucker in \cite[Theorem~A]{PolstraTucker}.

Relating $F$-signature with relative Hilbert-Kunz multiplicities is closely connected with the weak implies strong conjecture. Under mild hypotheses, a local ring $R$ is weakly $F$-regular if and only if $\e(I)-\e((I,u))>0$ for every $\fm$-primary ideal $I$ and $u\in R-I$ by \cite[Proposition~4.16 and Theorem~8.17]{HHJAMS}, and $R$ is strongly $F$-regular if and only if $s(R)>0$ by \cite{AberbachLeuschke}. In particular, if it is known that $F$-signature of a weakly $F$-regular ring can be realized as a relative Hilbert-Kunz multiplicity then the conjecture of weak implies strong would follow. The techniques of this article are used to equate $F$-signature with a relative Hilbert-Kunz multiplicity for strongly $F$-regular rings which satisfy the hypotheses of Theorem~\ref{Main tight closure theorem}.\footnote{The only property in the hypotheses of Theorem~\ref{Main tight closure theorem} which is not enjoyed by every strongly $F$-regular ring is the property that some symbolic power of the anti-canonical ideal has analytic spread at most $2$ on the punctured spectrum.}

\begin{Theoremx}\label{Main theorem prime char invariants}
Let $(R,\fm,k)$ be a local $F$-regular and $F$-finite ring of prime characteristic $p>0$ such that some symbolic power of the anti-canonical ideal has analytic spread at most $2$ on the punctured spectrum. There exists an irreducible $\fm$-primary ideal $I$ and socle generator $u$ mod $I$ such that 
\[
a_e(R)=\frac{\lambda(R/I^{[p^e]})}{p^{e\dim(R)}}-\frac{\lambda(R/(I,u)^{[p^e]})}{p^{e\dim(R)}}
\]
for every $e\in \N$. In particular,
\[
s(R)=\e(I)-\e((I,u)).
\]
\end{Theoremx}

Section~\ref{background section} is devoted to background and preliminary results. Central to this paper is the notion of a local cohomology bound. Local cohomology bounds are of independent interest and are defined and discussed in Section~\ref{lcb introduction}. Section~\ref{section weak implies strong method} is the technical heart of the paper and is where  proofs of Theorem~\ref{Main theorem dimension 4} and Theorem~\ref{Main tight closure theorem} are given. The proof of Theorem~\ref{Main theorem prime char invariants} is found in Section~\ref{section $F$-signature and relative Hilbert-Kunz multiplicity}. In Section~\ref{section open problems} we list some open problems of interest.

\section{Background and preliminary results}\label{background section}

\subsection{Tight closure}
 Let $R$ be a ring of prime characteristic $p>0$ and let $R^\circ$ be complement of the union of the minimal primes of $R$. The $e$th Frobenius functor, or the $e$th Peskine-Szpiro functor, is the functor $F^e:\Mod(R)\to \Mod(R)$ obtained by extending scalars along the $e$th iterate of the Frobenius endomorphism. If $N\subseteq M$ are $R$-modules and $m\in M$, then $m$ is in the tight closure of $N$ relative to $M$ if there exists a $c\in R^\circ$ such that for all $e\gg0$ the element $m$ is in the kernel of the following composition of maps:
  \[
 M\to M/N\to F^e(M/N)\xrightarrow{\cdot c} F^e(M/N).
 \]
 In particular, if we consider an inclusion of $R$-modules of the form $I\subseteq R$ then $F^e(R/I)\cong R/I^{[p^e]}$ where $I^{[p^e]}=(r^{p^e}\mid r\in I)$, and an element $r\in R$ is in the tight closure of $I$ relative to $R$ if there exists a $c\in R^\circ$ such that $cr^{p^e}\in I^{[p^e]}$ for all $e\gg 0$.  The tight closure of the module $N$ relative to the module $M$ is denoted $N^*_{M}$. In the case that $M=R$ and $N=I$ is an ideal then we denote the tight closure of $I$ relative to $R$ as $I^*$. We say that $N$ is tightly closed in $M$ if $N=N^*_M$. If an ideal is tightly closed in $R$ then we simply say that the ideal is tightly closed. The finitistic tight closure of $N\subseteq M$ is denoted $N^{*,fg}_M$ and is the union of $(N\cap M')^*_{M'}$ where $M'$ runs over all finitely generated submodules of $M$.

The notions of weak $F$-regularity and strong $F$-regularity can be compared by studying the finitistic tight closure and tight closure of the zero submodule of the injective hull of a local ring by \cite[Proposition~8.23]{HHJAMS} and \cite[Proposition~7.1.2]{SmithThesis}. Suppose that $(R,\fm,k)$ is complete local and $E_R(k)$ is the injective hull of the residue field. The finitistic test ideal of $R$ is $\tau_{fg}(R)=\bigcap_{I\subseteq R} \Ann_R(I^*/I)$ and agrees with $\Ann_R(0^{*,fg}_{E_R(k)})$. The test ideal of $R$ is $\tau(R)=\bigcap_{N\subseteq M\in \Mod(R)}\Ann_R(N^*_M/N)$ and agrees with $\Ann_R(0^*_{E_R(k)})$. The ring $R$ is weakly $F$-regular if and only if $\tau_{fg}(R)=R$ and $R$ is strongly $F$-regular if and only if $\tau(R)=R$. Thus to prove the conjectured equivalence of weak and strong $F$-regularity it is enough to show $0^*_{E_R(k)}=0^{*,fg}_{E_R(k)}$ under hypotheses satisfied by rings which are weakly $F$-regular.

To explore the tight closure of the zero submodule of $E_R(k)$ we exploit the structure of $E_R(k)$ as direct limit of $0$-dimensional Gorenstein quotients of $R$ described in \cite{HochsterPurity}. Suppose $(R,\fm,k)$ is a complete local Cohen-Macaulay domain of Krull dimension $d$ and $J_1\subsetneq R$ a canonical ideal. Let $0\neq x_1\in J_1$, $x_2,\ldots,x_d\in R$ a parameter sequence on $R/J_1$, and for each $t\in \N$ let $I_t=(x_1^{t-1}J_1,x_2^t,\ldots,x_d^t)$. The sequences of ideals $\{I_t\}$ form a decreasing sequence of irreducible $\fm$-primary ideals cofinal with $\{\fm^t\}$. Moreover, the direct limit  system $\varinjlim R/I_t\xrightarrow{\cdot x_1\cdots x_d}R/I_{t+1}$ is isomorphic to $E_R(k)$. There is flexibility in choosing parameters when realizing the injective hull as a direct limit just described and it will be beneficial to choose the parameter sequence to satisfy some additional properties.

\begin{definition} Let $(R, \fm, k)$ be a local catenary domain of dimension $d$, and let $J$ be an ideal of $R$ of pure height $1$.  We say that the sequence of elements $x_1, \ldots, x_d \in \fm$ is {\it suitable with respect to $J$}  (or merely {\it suitable}, if $J$ is clear) if
\begin{enumerate}
\item $x_1, \ldots, x_d$ is a system of parameter for $R$,
\item $x_ 1 \in J$ and $x_2, \ldots, x_d$ are parameters for $R/J$,
\item if $J_P$ is principal for all minimal primes of $J$, then $J_{x_2}$ is principal,
\item if $J$ is principal in codimension $2$, then $J_{x_3}$ is principal.
\end{enumerate}
\end{definition}

Observe that if $J\subseteq R$ is an ideal of pure height $1$ which is principal in codimension $2$ then there exists a parameter sequence which is suitable with respect to $J$.

\begin{lemma}\label{direct limit of finitisitc tight closure} Let $(R,\fm,k)$ be a complete Cohen-Macaulay local ring of prime characteristic $p>0$ and of Krull dimension $d$. Let $J_1\subsetneq R$ be a choice of canonical ideal and $x_1,\ldots, x_d$ a suitable system of parameters. Make the following identifications of $E_R(k)$ and $H^{d-1}_\fm(R/J_1)$:
\[
E_R(k)\cong \varinjlim \left(\frac{R}{(x_1^{t-1}J_1,x_2^t,\cdots x_d^t)}\xrightarrow{\cdot x_1\cdots x_d}\frac{R}{(x_1^{t}J_1,x_2^{t+1},\cdots x_d^{t+1})}\right)
\]
\[
H^{d-1}_\fm(R/J_1)\cong \varinjlim \left(\frac{R}{(J_1,x_2^t,\cdots x_d^t)}\xrightarrow{\cdot x_2\cdots x_d}\frac{R}{(J_1,x_2^{t+1},\cdots x_d^{t+1})}\right)
\]
Then under the above identifications of $E_R(k)$ and $H^{d-1}_\fm(R/J_1)$ we have that
\[
0^{*,fg}_{E_R(k)}\cong \varinjlim \left(\frac{(x_1^{t-1}J_1,x_2^t,\cdots x_d^t)^*}{(x_1^{t-1}J_1,x_2^t,\cdots x_d^t)}\xrightarrow{\cdot x_1\cdots x_d}\frac{(x_1^{t}J_1,x_2^{t+1},\cdots x_d^{t+1})^*}{(x_1^{t}J_1,x_2^{t+1},\cdots x_d^{t+1})}\right)
\]
and
\[
0^{*,fg}_{H^{d-1}_\fm(R/J)}\cong \varinjlim \left(\frac{(J_1,x_2^t,\cdots x_d^t)^*}{(J_1,x_2^t,\cdots x_d^t)}\xrightarrow{\cdot x_2\cdots x_d}\frac{(J_1,x_2^{t+1},\cdots x_d^{t+1})^*}{(J_1,x_2^{t+1},\cdots x_d^{t+1})}\right).
\]
\end{lemma}

\begin{proof} The containments $\supseteq$ are clear by definition of finitistic tight closure. The containments $\subseteq$ are also straight forward since the maps in the direct limit systems above are injective under the Cohen-Macaulay assumption.
\end{proof}

\begin{lemma}\label{Lemma Equiv way to check weak and strong} Let $(R,\fm,k)$ be a complete Cohen-Macaulay local normal domain of Krull dimension $d$ and of prime characteristic $p>0$.  Then $0^*_{E_R(k)}=0^{*,fg}_{E_R(k)}$ if and only if $0^*_{H^{d-1}_\fm(R/J_1)}=0^{*,fg}_{H_\fm^{d-1}(R/J_1)}$ for every choice of canonical ideal $J_1\subsetneq R$.
\end{lemma}

\begin{proof}
 Let $x_1,\ldots,x_d\in R$ be a suitable system of parameters with respect to $J_1$, and identify the injective hull of the residue field as
\[
E_R(k)=\varinjlim_{t}\left(\frac{R}{(x_1^{t-1}J_1,x_2^t,\ldots,x_d^t)}\xrightarrow{\cdot x_1\cdots x_d}\frac{R}{(x_1^{t}J_1,x_2^{t+1},\ldots,x_d^{t+1})}\right).
\]

Suppose that $0^*_{H^{d-1}_\fm(R/J_1)}=0^{*,fg}_{H^{d-1}_\fm(R/J_1)}$ for every $t\in\mathbb{N}$ for every choice of canonical ideal $J_1\subsetneq R$.  Let $\eta\in 0^*_{E_R(k)}$. Suppose that $\eta=r+(x_1^{s-1}J_1,x_2^s,\ldots,x_d^s)$. If we replace the canonical ideal $J_1$ with the canonical ideal $x_1^{s-1}J_1$ and the parameter sequence $x_1,x_2,\ldots,x_d$ with $x_1^{s},x_2^s,\ldots,x_d^s$ then we may suppose that $\eta=r+(J_1,x_2,\ldots,x_d)$. Then there exists a $c\in R^\circ$ such that $c\eta^{p^e}=0$ for all $e\geq 1$. Equivalently, for every $e\in \N$ there exists a $t\in \N$ such that 
\[
cr^{p^e}(x_1\cdots x_d)^{(t-1)p^e}\in (x_1^{t-1}J_1,x^t_2,\ldots,x^t_d)^{[p^e]},
\]
in which case there exists an element $s\in J_1^{[p^e]}$ such that 
\[
(cr^{p^e}(x_2\cdots x_d)^{(t-1)p^e}-s)x_1^{(t-1)p^e}\in (x_2^t,\ldots,x^t_d)^{[p^e]}.
\]
But $x_1,x_2,\ldots,x_d$ is a regular sequence and therefore 
\[
cr^{p^e}(x_2\cdots x_d)^{(t-1)p^e}-s\in (x_2^t,\ldots,x^t_d)^{[p^e]}
\] and hence 
\[
cr^{p^e}(x_2\cdots x_d)^{(t-1)p^e}\in (J_1,x^t_2,\ldots,x^t_d)^{[p^e]}.
\]
If we identify $H^{d-1}_\fm(R/J_1)$ as 
\[
\varinjlim\left(\frac{R}{(J_1,x_2^t,\ldots,x_d^t)}\xrightarrow{\cdot x_2\cdots x_d}\frac{R}{(J_1,x_2^{t+1},\ldots,x_d^{t+1})}\right)
\]
then the above shows that the class of $r+(J_1,x_2,\ldots ,x_d)$ is an element of $0^*_{H^{d-1}_\fm(R/J_1)}$. Moreover, under the direct limit identification of $H^{d-1}_\fm(R/J_1)$ we have by Lemma~\ref{direct limit of finitisitc tight closure} that
\[
0^*_{H^{d-1}_\fm(R/J_1)}=0^{*,fg}_{H^{d-1}_\fm(R/J_1)}\cong \varinjlim\left(\frac{(J_1,x_2^t,\ldots,x_d^t)^*}{(J_1,x_2^t,\ldots,x_d^t)}\xrightarrow{\cdot x_2\cdots x_d}\frac{(J_1,x_2^{t+1},\ldots,x_d^{t+1})^*}{(J_1,x_2^{t+1},\ldots,x_d^{t+1})}\right).
\]
In particular, there exists a $t\in \N$ such that $(x_2\cdots x_d)^t r \in (J_1,x_2^{t+1},\ldots,x_d^{t+1})^*$. It follows that $(x_1x_2\cdots x_d)^t r  \in (x_1^tJ_1,x_2^{t+1},\ldots,x_d^{t+1})^*$ and therefore $\eta\in 0^{*,fg}_{E_R(k)}$.

Conversely, suppose that $0^*_{E_R(k)}=0^{*,fg}_{E_R(k)}$ and let $\eta\in 0^*_{H^{d-1}_\fm(R/J_1)}$. As above, we may replace our choice of canonical ideal and parameters $x_2,\ldots,x_d$ and assume that $\eta=r+(J_1,x_2,\ldots,x_d)$. Then there exists a $c\in R^\circ$ such that for every $e\in \N$ there exists a $t\in \N$ such that 
\[
cr^{p^e}(x_2\cdots x_d)^{(t-1)p^e}\in (J_1,x_2^t,\cdots,x_d^t)^{[p^e]}.
\]
It then follows that 
\[
cr^{p^e}(x_1x_2\cdots x_d)^{(t-1)p^e}\in (x_1^{t-1}J_1,x_2^t,\cdots,x_d^t)^{[p^e]}
\]
and therefore the element $r+(J_1,x_2,\ldots, x_d)$ of $E_R(k)$ is an element of $0^*_{E_R(k)}$. Under the direct limit identification of $E_R(k)$ we have by Lemma~\ref{direct limit of finitisitc tight closure} that
\[
0^*_{E_R(k)}=0^{*,fg}_{E_R(k)}\cong \varinjlim\left(\frac{(x_1^{t-1}J_1,x_2^t,\cdots,x_d^t)^*}{(x_1^{t-1}J_1,x_2^t,\cdots,x_d^t)}\xrightarrow{\cdot x_1\cdots x_d}\frac{(x_1^{t}J_1,x_2^{t+1},\cdots,x_d^{t+1})^*}{(x_1^{t}J_1,x_2^{t+1},\cdots,x_d^{t+1})}\right).
\]
Therefore there exists a $t\in \N$ such that $(x_1\cdots x_d)^tr\in (x_1^{t}J_1,x_2^{t+1},\ldots,x_d^{t+1})^*$, i.e., there exists a $c\in R^\circ$ such that 
\[
c((x_1\cdots x_d)^tr)^{p^e}\in (x_1^{t}J_1,x_2^{t+1},\ldots,x_d^{t+1})^{[p^e]}
\]
for every $e\in \N$. Thus for every $e\in \N$ there exists a $s\in J^{[p^e]}$ such that 
\[
(c((x_2\cdots x_d)^tr)^{p^e}-s)x_1^{tp^e}\in (x_2^{t+1},\cdots x_d^{t+1})^{[p^e]}.
\]
But $x_1,\ldots, x_d$ is a regular sequence and it follows that \[
c((x_2\cdots x_d)^tr)^{p^e}\in (J_1,x_2^{t+1},\ldots,x_d^{t+1})^{[p^e]}
\] 
for every $e\in \N$. In particular, $(x_2\cdots x_d)^tr\in (J_1,x_2^{t+1},\ldots,x_d^{t+1})^*$ and therefore $
\eta=(x_2\cdots x_d)^tr+ (J_1,x_2^{t+1},\ldots,x_d^{t+1})$
is an element of $0^{*,fg}_{H^{d-1}_\fm(R/J_1)}$.
 \end{proof}
 
 \subsection{$S_2$-ification and higher Ext-modules}
 
 Though we do not directly use the results of \cite{DuttaI, DuttaII}, we would like to mention that important aspects of our techniques are inspired by these two articles. For example, suppose $(S,\fn,k)$ is a Cohen-Macaulay local domain of dimension $d$ and $M$ a finitely generated $S$-module such that $\Ht(\Ann_S(M))=h$. Let $(F_\bullet,\partial_\bullet)$ be the minimal free resolution of $M$, let $(-)^*$ denote $\Hom_S(-,S)$, and consider the dual complex $(F_\bullet^*,\partial_\bullet^*)$. Because  $\Ht(\Ann_S(M))=h$ we have that the following complex is exact:
 \[
 0\to F_0^*\xrightarrow{\partial_1^*}F_1^*\to \ldots\to F_{h-1}^*\xrightarrow{\partial_{h}^*} F_h^*\to \coker(\partial_{h}^*)\to 0. 
 \]
In particular, $\depth(\coker(\partial_{h}^*))=d-h$. Moreover, there is a short exact sequence
\[
0\to \Ext^{h}_S(M,S)\to \coker(\partial_h^*)\to \im(\partial_{h+1}^*)\to 0.
\]
The module $\im(\partial_{h+1}^*)$ is torsion-free and therefore has depth at least $1$. If $d-h\geq 2$ then $\Ext^{h}_S(M,S)$ has depth at least $2$. If $d-h=1$ then the depth of $\Ext^{h}_S(M,S)$ is $1$. If $d-h=0$ then $M$ is $0$-dimensional. Therefore if $\Ht(\Ann_S(M))=h$ then $\Ext^{h}_S(M,S)$ is an $(S_2)$-module over its support, an observation we record for future reference.

\begin{lemma}\label{lemma s2-module} Let $(S,\fm,k)$ be a Cohen-Macaulay local domain and $M$ a finitely generated $S$-module such that $\Ht(\Ann_S(M))=h$. Then $\Ext^{h}_S(M,S)$ is an $(S_2)$-module over its support.
\end{lemma}

Continue to consider the ring $S$, the module $M$, and the resolution $(F_\bullet,\partial_\bullet)$ as above. Also consider the minimal free resolution $(G_\bullet,\delta_\bullet)$ of $\Ext^{h}_S(M,S)$. If $\depth(M)=d-h$ is maximal, then $\Ext^{h}_S(M,S)=\coker(\partial_h^*)$ and therefore $(G_\bullet,\delta_\bullet)$ is the complex
\[
 0\to F_0^*\xrightarrow{\partial_1^*}F_1^*\to \ldots\to F_{h-1}^*\xrightarrow{\partial_{h}^*} F_h^*\to 0.
\]
In particular, if $\depth(M)=d-h$ then $\Ext^{h}_S(\Ext^h_S(M,S),S)\cong M$. Suppose $\depth(M)<d-h$ and let $(F_\bullet^*,\partial_\bullet^*)_{tr}$ be the complex obtained by truncating $(F_\bullet^*,\partial_\bullet^*)$ at the $h$th spot. That is $(F_\bullet^*,\partial_\bullet^*)_{tr}$ is the minimal free resolution of $\coker(\partial_h^*)$. Then the natural inclusion $\Ext^h_S(M,S)\subseteq \coker(\partial_h^*)$ lifts to a map of complexes $(G_\bullet,\delta_\bullet)\to (F_\bullet^*,\partial_\bullet^*)_{tr}$ and therefore there is an induced map $M\to \Ext^h_S(\Ext^h_S(M,S),S)$.

\begin{lemma}\label{Inclusion lemma} Let $(R,\fm,k)$ be a complete local domain of dimension at least $3$ and $J\subseteq R$ a pure height $1$ ideal. Suppose $(S,\fn,k)$ is a regular local ring mapping onto $R$, $R\cong S/P$, and $\Ht(P)=h$. Then for every integer $i$ the kernel of the natural map $R/J^i\to \Ext^{h+1}_S(\Ext^{h+1}_S(R/J^i,S),S)$ is $J^{(i)}/J^i$. In particular, for every integer $i$ there is a natural inclusion $R/J^{(i)}\subseteq \Ext^{h+1}_S(\Ext^{h+1}_S(R/J^i,S),S)$. Moreover, the natural inclusion $R/J^{(i)}\subseteq \Ext^{h+1}_S(\Ext^{h+1}_S(R/J^i,S),S)$ is an isomorphism whenever localized at prime ideal $\fp\in V(J)$ such that $(R/J^{(i)})_\fp$ is Cohen-Macaulay.
 \end{lemma}
\begin{proof}
It only remains to show that the kernel of $R/J^i\to \Ext^{h+1}_S(\Ext^{h+1}_S(R/J^i,S),S)$ is $J^{(i)}/J^i$. But this follows from the observation that the map 
\[
R/J^i\to \Ext^{h+1}_S(\Ext^{h+1}_S(R/J^i,S),S)
\]
is an isomorphism when localized at any minimal component of $J$ by the discussion proceeding the statement of the lemma.
 \end{proof}
 
 We record a corollary of Lemma~\ref{Inclusion lemma} for future reference.
 
 \begin{corollary}\label{Inclusion lemma corollary}
 Let $(R,\fm,k)$ be a complete local Cohen-Macaulay domain, which is $\Q$-Gorenstein in codimension $2$, and $J_1\subsetneq R$ a choice of canonical ideal. Let $m\in\N$ be an integer such that $J_1^{(m)}$ is principal in codimension $2$. Suppose $(S,\fn,k)$ is a regular local ring mapping onto $R$, $R\cong S/P$, and $\Ht(P)=h$. Then for every integer $i$ the natural inclusion $R/J_1^{(mi+1)}\to \Ext^{h+1}_S(\Ext^{h+1}_S(R/J_1^{mi+1},S),S)$ is an isomorphism whenever localized at a prime ideal of $R$ of height $2$ or less.
 \end{corollary}
 \begin{proof}
Immediate by Lemma~\ref{Inclusion lemma} since $J_1^{(mi+1)}R_\fp\cong J_1R_\fp$ is a canonical ideal whenever $\fp$ is a prime of $R$ of height $2$ or less.
 \end{proof}

\subsection{Rees algebras, symbolic Rees algebras, and analytic spread}

Let $R$ be a Noetherian domain and $I\subseteq R$ an ideal. The Rees ring of $I$ is the blowup algebra
\[
R[It]=R\oplus I\oplus I^2\oplus \cdots.
\]
If all associated primes of $I$ are minimal and $W$ is the complement of the union of the prime components of $I$, then the $N$th symbolic power of $I$ is the ideal $I^{(N)}=I^NR_W\cap R$. The symbolic Rees ring of $I$ is the $R$-algebra
\[
\mathcal{R}_I:=R\oplus I\oplus I^{(2)}\oplus \cdots,
\]
an $R$-algebra with the potential of being non-Noetherian, \cite{CutkoskyDuke, Rees58, Roberts}. We will typically be interested symbolic Rees rings associated to ideals of pure height $1$.

Suppose further that $(R,\fm,k)$ is local. Then the analytic spread of $I$ is the Krull dimension of the fiber cone 
\[
k\otimes_R R[It]\cong k\oplus \frac{I}{\fm I}\oplus \frac{I^2}{\fm I^2}\oplus \cdots.
\]
The analytic spread of a nonzero proper ideal $I$ is a natural number between $1$ and $\dim(R)$. If all associated primes of the ideal $I$ are minimal and the symbolic Rees ring $\mathcal{R}_I$ is Noetherian then we can compare the analytic spread of $I^{(N)}$ with the analytic spread of $I$.

 \begin{proposition}\label{Proposition analytic spread of symbolic powers} Let $(R,\fm,k)$ be a excellent local Noetherian normal domain and $I\subseteq R$ an ideal without embedded components. Suppose that the analytic spread of $IR_P$ is no more than $\Height(P)-1$ for each prime $P\supseteq I$ which is not an associated prime of $I$. If the analytic spread of $I$ is $\ell$ then for each integer $N\in\N$ the analytic spread of $I^{(N)}$ is no more than $\ell$.
\end{proposition}

\begin{proof} Under the assumptions of the proposition  the symbolic Rees ring $\mathcal{R}_I$ is a graded subalgebra of the normalization of $R[It]$, \cite[Theorem~1.1]{CutkoskyHerzogSrinivasan}. In particular, $R[It]\to \mathcal{R}_I$ is finite. Hence the maps of the $N$th Veronese subalgebras
\[
R[I^Nt]\to \mathcal{R}_{I^{(N)}}
\]
are finite for each integer $N$. Observe that the $R$-algebra map above can be factored as
\[
R[I^Nt]\to R[I^{(N)}t]\to \mathcal{R}_{I^{(N)}} .
\]
Therefore the induced map of fiber cones
\[
k\otimes R[I^Nt]\to k\otimes R[I^{(N)}t]
\]
are finite for each integer $N$. In particular, $k\otimes R[I^{(N)}t]$ has Krull dimension no more than the analytic spread of $I^N$ and the analytic spread of $I^N$ is equal to the analytic spread of $I$. \footnote{Jonathan Monta\~no has shown to us Proposition~\ref{Proposition analytic spread of symbolic powers} can be significantly generalized. It is possible to adapt the proof technique of Proposition~\ref{Proposition analytic spread of symbolic powers} under the weaker assumptions that $R$ is assumed to be a domain which is analytically unramified and formally equidimensional. Under these assumptions the normalization of $R[It]$ is Noetherian and one can adapt the proof of \cite[Theorem~1.1]{CutkoskyHerzogSrinivasan} to this scenario.}
\end{proof}

Finite generation of symbolic Rees rings away from the maximal ideal of a local ring allows us to effectively compare ordinary and symbolic powers of an ideal.

\begin{proposition}\label{Proposition finite generation on punctured spectrum} Let $(R,\fm,k)$ be a local domain and $I\subseteq R$ an ideal without embedded components. Suppose that for each $P\in \Spec(R)-\{\fm\}$ that $R_P\otimes \mathcal{R}_I$ is a Noetherian $R_P$-algebra. Then there exists an integer $N\in\N$ such that for all $i\in \N$ the inclusion of ideals $I^{(N)i}\subseteq I^{(Ni)}$ agree when localized at any point of $\Spec(R)-\{\fm\}$.
\end{proposition}

\begin{proof} Let $P\in\Spec(R)- \{\fm\}$. We aim to show that there exists a natural number $N_P$, depending on $P$, and open neighborhood $U_P- \Spec(R)\setminus\{\fm\}$ of $P$ so that for all $Q\in U$ we have that $R_Q\otimes \mathcal{R}_{I^{(N_P)}}$ is a standard graded $R_Q$-algebra. Once this is accomplished we can then cover $\Spec(R)-\{\fm\}$ by finitely many open sets $U_{P_1},\ldots, U_{P_t}$ so that for each $1\leq i\leq t$ the algebra $R_{Q}\otimes \mathcal{R}_{I^{(N_{P_i})}}$ is a standard graded $R_{Q}$-algebra for each $Q\in U_{P_i}$. We then take $N$ to be a common multiple of $N_1,N_2,\ldots, N_t$ so that $\mathcal{R}_{I^{(N)}}$ loaclizes to a standard graded $R_P$-algebra for each $P\in\Spec(R)-\{\fm\}$, i.e. for each $i\in\mathbb{N}$ the inclusion of ideals $I^{(N)i}\subseteq I^{(Ni)}$ agrees when localized at any point of $\Spec(R)-\{\fm\}$.

To this end we let $P\in\Spec(R)-\{\fm\}$. We are assuming that $R_P\otimes \mathcal{R}_I$ is a Noetherian $R_P$-algebra. Thus there exists an integer $N$ so that the $N$th Veronese of $R_P\otimes \mathcal{R}_I$ is a standard graded $R_P$-algebra, i.e. $R_P\otimes \mathcal{R}_{I^{(N)}}$ is a standard graded $R_P$-algebra. Equivalently, for each $i\in \mathbb{N}$ the inclusion of ideals $I^{(N)i}\subseteq I^{(Ni)}$ agrees when localized at $P$. Now we consider the collection of associated primes $\Lambda=\bigcup_{i\in \mathbb{N}}\Ass\left(I^{(N)i}\right)$. The set $\Lambda$ is a finite set by \cite{Brodmann}, see also \cite{HunekeSmrinov}. By prime avoidance we can choose an element $s\in R$ which is contained in each non-minimal member of $\Lambda$ and avoids each of the prime components of $I$. If $Q$ is a prime ideal in the open set $D(s)=\{Q\in\Spec(R)\mid s\not\in Q\}$ then for each $i\in\mathbb{N}$ the collection of associated primes $\bigcup_{i\in\mathbb{N}}I^{(N)i}R_Q$ agrees with the collection of minimal primes of $I^{(N)}R_Q$, i.e. the inclusion of ideals $I^{(N)i}\subseteq I^{(Ni)}$ agrees when localized at any prime of $D(s)$. 
\end{proof}

Let $I\subseteq R$ be an ideal whose components have the same height and consider the finite collection of associated primes of the set of ideals $\{I^n\}_{n\in \N}$. The finite set of associated primes of the collection of ideals $\{I^n\}_{n\in \N}$ are known as the asymptotic associated primes of $I$. Suppose that $P_1,\ldots, P_n$ are the finitely many non-minimal asymptotic associated primes of $I$ and let $\fa=P_1\cap \cdots \cap P_n$. Then for each integer $N\in \N$ we have that $I^{(N)}=(I^N:\fa^{\infty}):=\{r\in R\mid \fa ^i r \subseteq I^N\, \forall i\gg 0\}$. The analytic spreads of the collection of ideals $\{I^{(N)}R_{P}\}_{N\in \mathbb{N}, P\in V(\fa)}$ and finite generation of the symbolic Rees ring $\mathcal{R}_I$ have an interesting connection.

\begin{theorem}{\cite[Theorem~1.1 and Theorem~1.5]{CutkoskyHerzogSrinivasan}}
\label{Criterion for finite generation}
Let $R$ be an excellent Noetherian normal domain of Krull dimension $d$ and $I\subseteq R$ an ideal without embedded components. Suppose $\fa\subseteq R$ is a reduced ideal of height at least $2$. Then the following are equivalent:
\begin{enumerate}
\item The ring $\displaystyle \bigoplus (I^N:\fa^\infty)$ is Noetherian;
\item There exists an integer $m$ so that for all $P\in V(\fa)$ the analytic spread of $(I^m:\fa^\infty)R_P$ is no more than $\Ht P-1$;
\item There exists an integer $\ell$ so that if $J=(I^\ell:\fa^\infty)$ then there is a containment of $R$-algebras $\displaystyle \bigoplus_N (J^{N}:\fa^\infty)\subseteq \overline{R[Jt]}$ where $\overline{R[Jt]}$ is the normalization of the Rees ring $R[Jt]$.
\end{enumerate}
In particular, if $\fa$ is the intersection of the non-minimal asymptotic primes of $I$ then the symbolic Rees ring $\mathcal{R}_I$ is Noetherian if and only if there exists an integer $m\in \N$ such that the analytic spread of $I^{(m)}R_P$ is no more than $\Ht P-1$ at each $P\in V(\fa)$.
\end{theorem}

The criterion described in Theorem~\ref{Criterion for finite generation} to determine finite generation of symbolic Rees rings is originally due to Katz and Ratliff, \cite[Theorem~A and Corollary~1]{KatzRatliff}. The reader interested in learning more about connections between finite generation of symbolic Rees rings and analytic spread will also be interested in \cite{Schenzel} and \cite{DaoMontano}. We also remark that finite generation of symbolic Rees rings is deeply rooted to progress in the minimal model program. This is because finite generation of certain symbolic Rees rings is equivalent to the existence of flips, see \cite[Lemma~6.2 and Remark~6.3]{KollarMori}.

The following is a consequence of Theorem~\ref{Criterion for finite generation} and will be used in Section~\ref{section weak implies strong method}.

\begin{proposition}
\label{Proposition Noetherian if analytic spread 2} Let $R$ be an excellent Noetherian normal domain. Suppose that $I\subseteq R$ an ideal of pure height $1$ with analytic spread at most $2$ and suppose that as an element of the divisor class group of $R$ the ideal $I$ is torsion in codimension $2$. Then the symbolic Rees ring $\mathcal{R}_I$ is Noetherian.
\end{proposition} 

\begin{proof}
Let $\fa$ be the intersection of the asymptotic primes of $I$ of height at least $3$ and let $\fb$ be the intersection of the asymptotic primes of $I$ of height $2$. Then the $N$th symbolic power of the ideal $I$ is realized as $(I^N:\fa^\infty):\fb^{\infty}$. The analytic spread of $I$ is at most $2$ and the analytic spread of $I$ does not increase under localization. Therefore the $R$-algebra $\bigoplus (I^N:\fa^\infty)$ is Noetherian by Theorem~\ref{Criterion for finite generation}. Hence there exists an integer $m\in \N$ such that $\bigoplus (I^{mN}:\fa^\infty)$ is a standard graded $R$-algebra. Equivalently, for each integer $N\in \N$ we have that $(I^{m}:\fa^\infty)^N=(I^{mN}:\fa^\infty)$. Let $\tilde{I}=I^m:\fa^\infty$. Because we are assuming $I$ is torsion as an element in the divisor class group in codimension $2$ we can choose an integer $n$ such that $I^{(mn)}=\tilde{I}^n:\fb^\infty$ is principal in codimension $2$. In particular, the analytic spread of $\tilde{I}^n:\fb^\infty$ is $1$ at each of the components of $\fb$. Therefore the symbolic Rees ring $\mathcal{R}_{I^{(mn)}}=\bigoplus_{N\in \N} \tilde{I}^{nN}:\fb^\infty$ is Noetherian by a second application of Theorem~\ref{Criterion for finite generation}. It then follows that the symbolic Rees ring $\mathcal{R}_I$ is Noetherian since the $mn$th Veronese embedding of $\mathcal{R}_I$ is Noetherian, see the proof of \cite[Theorem~2.1]{HerzogHibiTrung}.
\end{proof}

Another important concept surrounding the theory of analytic spread and reductions is the notion of a reduction number. Let $R$ be a Noetherian ring and $J\subseteq I$ ideals such that $J$ forms a reduction of $I$. The reduction number of $I$ with respect $J$ is the least integer $N$ such that $JI^N=I^{N+1}$. A theorem of Hoa allows us to relate reduction numbers with the analytic spread of an ideal via understanding properties of the graded ring $\gr_I(R)=\oplus I^{i}/I^{i+1}$. But first, recall that if $S=S_0\oplus S_1\oplus \cdots $ is a graded ring and $S_+$ is the irrelevant ideal then the $i$th $a$-invariant of $S$ is denoted by $a_i(S)$ and is the largest degree of support of the local cohomology module $H^i_{S_+}(S)$.

\begin{theorem}{\cite[Theorem~2.1]{Hoa}}\label{Hoa's theorem on reduction numbers} Let $(R,\fm,k)$ be a Noetherian local ring and $I\subseteq R$ an ideal. Let $\ell$ be the analytic spread of $I$ and suppose that $a_\ell(\gr_I(R))<0$. Then for all integers $n\gg0$ and reductions $J$ of $I^n$ the ideal $I^n$ has reduction number with respect to $J$ no more than $\ell-1$.
\end{theorem}

As a consequence to Theorem~\ref{Hoa's theorem on reduction numbers} we can effectively estimate the reduction numbers of large powers of pure height $1$ ideals of a strongly $F$-regular ring.

\begin{theorem}\label{Linquan's Theorem}
Let $(R,\fm,k)$ be a strongly $F$-regular and $F$-finite local ring of prime characteristic $p>0$ and dimension $d\geq 2$. Suppose further that $I\subseteq R$ is an ideal of pure height $1$  with the property that $I^n=I^{(n)}$ for all $n\in \mathbb{N}$. If $I$ has analytic spread $\ell\geq 2$ then for all $n\gg0$ the reduction number of $I^n$ with respect to any reduction is no more than $\ell-1$.
\end{theorem}

\begin{proof}
By Theorem~\ref{Hoa's theorem on reduction numbers} it is enough to show that $a_\ell(\gr_I(R))<0$. In fact, we will show that $a_i(\gr_I(R))<0$ for all $2\leq i \leq d$. But first, we will show $a_i(R[It])<0$ for all $2\leq i \leq d$. Because $R[It]=\mathcal{R}_I$ we have that $S:=R[It]$ is a strongly $F$-regular graded $R$-algebra by \cite[Lemma~3.1]{CEMS}, see also \cite[Theorem~0.1]{WatanabeCyclicCover} and \cite[Main Theorem]{MaPolstraSchwedeTucker}. The cohomology groups $H^i_{S_+}(S)$ are only supported in finitely many positive degrees. Indeed, let $X=\Proj(S)$ so that $H^{i}_{S_+}(S)\cong H^{i-1}(X,\mathcal{O}_X)$ for all $i\geq 2$, see \cite[Theorem~12.41]{24hours}, and therefore $[H^{i}_{S_+}(S)]_N=H^{i-1}(X,\mathcal{O}_X(N))=0$ for all $N\gg 0$ by Serre vanishing, \cite[Theorem~5.2]{Hartshorne}. It follows that there exists a homogeneous positive degree element $c\in S$ such that $c[H^i_{S_+}(S)]_{\geq 0}=0$. Because $S$ is strongly $F$-regular the $S$-linear maps $S\xrightarrow{\cdot F^e_*c} F^e_*S$ are pure for all $e\gg 0$. Therefore the $e$th Frobneius action on $H^i_{S_+}(S)$ followed by multiplying by $c$, which is the map realized by tensoring the pure map $S\xrightarrow{\cdot F^e_*c} F^e_*S$ with $H^i_{S_+}(S)$, are injective. But the $e$th Frobenius action of $H^i_{S_+}(S)$ maps elements of degree $n$ to elements of degree $np^e$. Furthermore, $c$ was chosen to annihilate elements of non-negative degree and therefore $H^i_{S_+}(S)$ can only be supported in negative degree.

The ring $S=R[It]$ is Cohen-Macaulay and therefore $a_d(\gr_I(R))<0$ by \cite[Theorem~3.1]{Hoa}. By \cite[Theorem~3.1 (ii)]{Trung} we have that $a_i(\gr_I(R))=a_{i}(S)$ whenever $a_i(\gr_I(R))\geq a_{i+1}(\gr_I(R))$. An easy descending induction argument now tells us that $a_i(\gr_I(R))<0$ for all $2\leq i \leq d$ and this completes the proof of the theorem.
\end{proof}

\section{Koszul cohomology, local cohomology, and local cohomology bounds}\label{lcb introduction}
 
In this section $R$ denotes a commutative Noetherian ring. Unless stated otherwise, we do not make any assumptions on the characteristic of $R$. Our study of local cohomology modules is centered around the realization of local cohomology as a direct limit system of Koszul cohomologies. We are interested in understanding at what point in a direct limit system that an element of a Koszul cohomology group representing the zero element of a local cohomology group becomes zero. Key to our study of local cohomology is the notion of a local cohomology bound relative to a sequence of elements defined below in Definition~\ref{Key definition}.  

\subsection{Definition of local cohomology bound}

Suppose $M$ is a module over a ring $R$ and $\underline{x}=x_1,\ldots ,x_d$ a sequence of elements. Then for each integer $j\in \N$ we let $\underline{x}^j=x_1^j,\ldots, x_d^j$ and for each pair of integers $j_1\leq j_2$ let $\tilde{\alpha}^\bullet_{M;\underline{x};j_1;j_2}$ denote the natural map of Koszul cocomplexes
\[
K^\bullet(\underline{x}^{j_1};M)\xrightarrow{\tilde{\alpha}^\bullet_{M;\underline{x};j_1;j_2}} K^\bullet(\underline{x}^{j_2};M).
\]
The map of cocomplexes $\tilde{\alpha}^\bullet_{M;\underline{x};j_1;j_2}$ is realized as the following tensor product of maps of Koszul cocomplexes on one element:
\[
\tilde{\alpha}^\bullet_{M;\underline{x};j_1;j_2}\cong \tilde{\alpha}^\bullet_{R;x_1;j_1;j_2}\otimes \tilde{\alpha}^\bullet_{R;x_2;j_1;j_2}\otimes \cdots \otimes \tilde{\alpha}^\bullet_{R;x_d;j_1;j_2}\otimes M.
\]
We let $\alpha^i_{M;\underline{x};j_1;j_2}$ denote the induced map of Koszul cohomologies
\[
H^i(\underline{x}^{j_1};M)\xrightarrow{\alpha^i_{M;\underline{x};j_1;j_2}} H^i(\underline{x}^{j_2};M).
\]
More specifically, suppose $j_1 = j$ and $j_2=j+k$ and consider the Koszul cocomplexes $K^\bullet(\underline{x}^j;M)$ and $K^\bullet(\underline{x}^{j+k};M)$. Then the cokernel of the $d$th map of these cocomplexes are $M/(\underline{x}^j)M$ and $M/(\underline{x}^{j+k})M$ respectively. Let $\alpha_{j,k}^\bullet: K^\bullet(\underline{x}^{j};M)\to K^{\bullet}(\underline{x}^{j+k};M)$ be the natural choice of map of cocomplexes lifting the map $M/(\underline{x}^j)M\xrightarrow{\cdot (x_1\cdots x_d)^k}M/(\underline{x}^{j+k})M$. Then $\alpha_{M;\underline{x};j;j+k}$ is the induced map $\alpha^i_{j,k}$ on Koszul cohomology.
In particular, 
\[
\varinjlim_{j_1\leq j_2} \left(H^i(\underline{x}^{j_1};M)\xrightarrow{\alpha^i_{M;\underline{x};j_1;j_2}} H^i(\underline{x}^{j_2};M)\right)\cong H^i_{(\underline{x})A}(M)
\]
by \cite[Theorem~3.5.6]{BrunsHerzog}.

Denote by $\alpha^i_{M;\underline{x};j;\infty}$ the natural map
\[
H^i(\underline{x}^{j};M)\xrightarrow{\alpha^i_{M;\underline{x};j;\infty}}H^i_{(\underline{x})A}(M).
\]
Observe that $\eta\in \ker (\alpha^i_{M;\underline{x};j;\infty})$ if and only if there exists some $k\geq 0$ such that $\eta\in \ker(\alpha^i_{M;\underline{x};j;j+k}).$ If $\eta\in  \ker (\alpha^i_{M;\underline{x};j;\infty})$ we let
\[
\epsilon_{\underline{x}^j}^i(\eta)=\min\{k\mid \eta \in  \ker (\alpha^i_{M;\underline{x};j;j+k})\}.
\]

\begin{definition}\label{Key definition} Let $R$ be a ring, $\underline{x}=x_1,\ldots, x_d$ a sequence of elements in $R$, and $M$ an $R$-module. The $i$th local cohomology bound of $M$ with respect to the sequence of elements $\underline{x}$ is 
\[
\lcb_i(\underline{x};M)=\sup\{\epsilon_{\underline{x}^j}^i(\eta)\mid \eta \in  \ker (\alpha^i_{M;\underline{x};j;\infty})\mbox{ for some }j \}\in \N\cup \{\infty\}.
\]
\end{definition}

Observe that if $M$ is an $R$-module and $\underline{x}$ is a sequence of elements, then $\lcb_i(\underline{x};M)=N<\infty$ simply means that if $\eta\in H^i(\underline{x}^j;M)$ represents the $0$-element in the direct limit 
\[
\varinjlim_{j_1\leq j_2} \left(H^i(\underline{x}^{j_1};M)\xrightarrow{\alpha^i_{M;\underline{x};j_1;j_2}} H^i(\underline{x}^{j_2};M)\right)\cong H^i_{(\underline{x})A}(M)
\]
then $\alpha^i_{M;\underline{x};j;j+N}(\eta)$ is the $0$-element of the Koszul cohomology group $H^{i}(\underline{x}^{j+N};M)$. Therefore finite local cohomology bounds correspond to a uniform bound of annihilation of zero elements in a choice of direct limit system defining a local cohomology module. It would be interesting to know when local cohomology bounds are finite.

\subsection{Basic properties of local cohomology bounds} Our study of local cohomology bounds begins with two elementary, yet useful, observations.

\begin{lemma}\label{lemma lcb power of elements} Let $R$ be a commutative Noetherian ring, $M$ an $R$-module, and $\underline{x}=x_1,\ldots, x_d$ a sequence of elements, then $\lcb_{i}(\underline{x}^j;M)\leq \lcb_i(\underline{x};M)$. Furthermore, $\lcb_i(\underline{x};M)\leq jm$ for some integers $j,m$ if and only if $\lcb_{i}(\underline{x}^j;M)\leq m$ where $\underline{x}^j$ is the sequence of elements $x_1^j,\ldots,x_d^j$.
\end{lemma}

 \begin{proof} One only has to observe that $\alpha^i_{M;\underline{x}^j;k,k+m}=\alpha^i_{M;\underline{x};jk,jk+jm}$.
 \end{proof}

If $x_1,\ldots, x_d$ is a sequence of elements in a ring $R$ and if $x_1M=0$ for some $R$-module $M$ then the short exact sequence of Koszul cocomplexes
\[
0\to K^\bullet(x_2,\ldots,x_d;M)(-1)\to K^\bullet(x_1,x_2,\ldots,x_d;M)\to K^\bullet(x_2,\ldots,x_d;M)\to 0
\]
is split and therefore $H^i(x_1,x_2,\ldots, x_d;M)\cong H^i(x_2,\ldots, x_d;M)\oplus H^{i-1}(x_2,\ldots, x_d;M)$. The content of the following lemma is a description of the behavior of the maps $\alpha^{i}_{M;x_1,x_2,\ldots,x_d; j,j+k}$ with respect to these isomorphisms of Koszul cohomologies.

\begin{lemma}\label{lemma on koszul complex containing 0 element}
Let $R$ be a commutative Noetherian ring, $M$ an $R$-module, and $x_1,x_2,\ldots,x_d$ a sequence of elements such that $x_1M=0$. If $i,j,k\in \N$ then
\[
H^i(x_1^j,x_2^j,\ldots,x_d^j; M)\cong H^i(x_2^j,\ldots,x_d^j;M)\oplus H^{i-1}(x_2^j,\ldots,x_d^j;M)
\]
and the map $\alpha_{M;x_1,x_2,\ldots, x_d;j,j+k}$ is the direct sum of $\alpha^{i}_{M;x_2,\ldots,x_d; j,j+k}$ and the $0$-map.
\end{lemma}

\begin{proof}
Let $(F^\bullet,\partial^\bullet)$ be the Koszul cocomplex $K^\bullet(x_2^j,\ldots,y_\ell^j;R)$ and let $(G^\bullet,\delta^\bullet)$ be the Koszul cocomplex $K^\bullet(x_1^j;R)$. Let
\[
(L^\bullet,\gamma^\bullet)=K^\bullet(x_1^j,x_2^j,\ldots,x_d^j;R)\cong K^\bullet(x_2^j,\ldots,x_d^j;R)\otimes K^\bullet(x_1^j;R).
\]
Then $L^i\cong (F^i\otimes G^0)\oplus (F^{i-1}\otimes G^1)\cong F^i\oplus F^{i-1}$. We abuse notation and let $\cdot x_1^j$ denote the multiplication map on $F^i$. Then up to sign on $\cdot x_1^j$  the map $\epsilon^i$ can be thought of as
\[
\gamma^i=\begin{pmatrix}\partial^i & 0\\ \cdot x_1^j & \partial^{i-1}\end{pmatrix}:F^i\oplus F^{i-1}\to F^{i+1}\oplus F^{i}.
\] 
In particular, if we apply $-\otimes_R M$ the map $\cdot x_1^j\otimes M$ is the $0$-map and therefore $i$th map of the Koszul cocomplex $K^i(x_1^j,x_2^j,\ldots,x_d^j;M)$ is the direct sum of maps $(\partial^i\otimes M)\oplus (\partial^{i-1}\otimes M)$. In particular
\[
H^i(x_1^j,x_2^j,\ldots,x_d^j; M)\cong H^i(x_2^j,\ldots,x_d^j;M)\oplus H^{i-1}(x_2^j,\ldots,x_d^j;M).
\]
To see that $\alpha_{M;x_1,x_2,\ldots, x_d;j,j+k}$ is the direct sum of $\alpha^{i}_{M;x_2,\ldots,x_d; j,j+k}$ and the $0$-map is similar to above argument but uses the fact that
\[
\tilde{\alpha}^\bullet_{M;x_1,x_2,\ldots,x_d;j;j+k}=\tilde{\alpha}^\bullet_{R;x_2,\ldots,x_d;j;j+k}\otimes \tilde{\alpha}^\bullet_{R;x_1;j;j+k}\otimes M
\]
and $ \tilde{\alpha}^1_{R;x_1;j;j+k}\otimes M=0$.
\end{proof}

A particularly useful corollary of Lemma~\ref{lemma on koszul complex containing 0 element} is the following:

 \begin{corollary}\label{corollary 0 map of Koszul cohomology} Let $R$ be a commutative Noetherian ring and $M$ an $R$-module. Suppose $x_1,\ldots,x_d$ is a sequence of elements and $(x_1,\ldots,x_{d-i})M=0$. If $j,k\in \N$ then
 \[
 \alpha^\ell_{M;x_1,\ldots,x_d;j,j+k}:H^{\ell}(x^j_1,\ldots,x^j_d;M)\to  H^{\ell}(x^{j+k}_1,\ldots,x^{j+k}_d;M)
 \]
 is the $0$-map for all $\ell\geq i+1$. In particular, $\lcb_\ell(x_1,\ldots,x_d;M)=1$ for all $\ell \geq i+1$.
 \end{corollary}
 
 \begin{proof}
By multiple applications of Lemma~\ref{lemma on koszul complex containing 0 element} it is enough to observe that
\[
H^{\ell}(x^j_{d-i+1},\ldots, x_d^j;M)=0.
\]
This is clearly the case since $x^j_{d-i+1},\ldots, x_d^j$ is a list of $i$ elements and we are examining an $\ell \geq i+1$ Koszul cohomology of $M$ with respect to this sequence.
 \end{proof}
 
Suppose $0\to M_1\to M_2\to M_3\to 0$ is a short exact sequence of $R$-modules. The next two properties of local cohomology bounds we record allow us to compare the local cohomology bounds of the modules appearing in the short exact sequence. Proposition~\ref{second proposition 0 map Koszul cohomology} allows us to effectively compare the local cohomology bounds of two of the terms in the sequence  provided a subset of the elements in the sequence of elements defining Koszul cohomology annihilates the third. Proposition~\ref{lcb and ses} compares the the local cohomology bounds of two of the terms in the short exact whenever the sequence of elements defining Koszul cohomology is a regular sequence on the third module.

 \begin{proposition}\label{second proposition 0 map Koszul cohomology} Let $(R,\fm,k)$ be a local ring and 
 \[
 0\to M_1\to M_2\to M_3\to 0
 \]
 a short exact sequence of finitely generated $R$-modules. Let $\underline{x}=x_1,\ldots,x_d$ be a sequence of elements of $R$.
 \begin{enumerate}
 \item If $(x_1,\ldots, x_{d-j})M_1=0$ then for all $\ell\geq j+1$
 \[
 \lcb_\ell(\underline{x};M_2)\leq \lcb_\ell(\underline{x};M_3)+1.
 \]
 \item If $(x_1,\ldots, x_{d-j})M_2=0$ then for all $\ell\geq j+1$
 \[
 \lcb_\ell(\underline{x};M_3)\leq \lcb_{\ell+1}(\underline{x};M_1)+1.
 \]
  \item If $(x_1,\ldots, x_{d-j})M_3=0$ then for all $\ell\geq j+1$
 \[
 \lcb_\ell(\underline{x};M_1)\leq \lcb_{\ell}(\underline{x};M_2)+1.
 \]
 \end{enumerate}
 \end{proposition}

 \begin{proof} For each integer $j\in \N$ let $\underline{x}^j$ denote the sequence of elements $x_1^j,x_2^j,\ldots, x_d^j$. For (1) we consider the following commutative diagram, whose middle row is exact:
 \[
 \begin{tikzcd}
\,&  H^\ell(\underline{x}^j;M_2)\arrow{r}\arrow{d}{\alpha^\ell_{M_2;\underline{x};j;j+k}}  &  H^\ell(\underline{x}^j;M_3)\arrow{d}{\alpha^\ell_{M_3;\underline{x};j;j+k}} \\
  H^{\ell}(\underline{x}^{j+k};M_1)\arrow{r}\arrow{d}{\alpha^\ell_{M_1;\underline{x};j+k;j+k+1}}  &  H^\ell(\underline{x}^{j+k};M_2)\arrow{r}\arrow{d}{\alpha^\ell_{M_2;\underline{x};j+k;j+k+1}} &  H^\ell(\underline{x}^{j+k};M_3)\\
   H^\ell(\underline{x}^{j+k+1};M_1)\arrow{r} & H^\ell(\underline{x}^{j+k+1};M_2)
 \end{tikzcd}
 \]
By Corollary~\ref{corollary 0 map of Koszul cohomology} the map $\alpha^\ell_{M_1;\underline{x};j+k;j+k+1} $ is the $0$-map for all $\ell\geq j+1$. A straightforward diagram chase of the above diagram, which follows an element $\eta\in \ker(\alpha^\ell_{M_2;\underline{x};j;j+k'})$ for some $k'$, shows that $\eta\in \ker(\alpha^\ell_{M_2;\underline{x};j;j+k+1})$ whenever $k\geq \lcb_\ell(\underline{x};M_3)$. In particular, $\lcb_\ell(\underline{x};M_2)\leq  \lcb_\ell(\underline{x};M_3)+1$.

Statements $(2)$ and $(3)$ follow in a similar manner. For $(2)$ one needs to consider the commutative diagrams
 \[
 \begin{tikzcd}
\,&  H^\ell(\underline{x}^j;M_3)\arrow{r}\arrow{d}{\alpha^\ell_{M_3;\underline{x};j;j+k}}  &  H^{\ell+1}(\underline{x}^j;M_1)\arrow{d}{\alpha^{\ell+1}_{M_1;\underline{x};j;j+k}} \\
  H^{\ell}(\underline{x}^{j+k};M_2)\arrow{r}\arrow{d}{\alpha^\ell_{M_2;\underline{x};j+k;j+k+1}}  &  H^\ell(\underline{x}^{j+k};M_3)\arrow{r}\arrow{d}{\alpha^\ell_{M_3;\underline{x};j+k;j+k+1}} &  H^{\ell+1}(\underline{x}^{j+k};M_1)\\
   H^\ell(\underline{x}^{j+k+1};M_2)\arrow{r} & H^\ell(\underline{x}^{j+k+1};M_3)
 \end{tikzcd}
 \]
 and invoke Corollary~\ref{corollary 0 map of Koszul cohomology} to know that $\alpha^\ell_{M_2;\underline{x};j+k;j+k+1}$ is the $0$-map for all $\ell\geq j+1$.
 
 For $(3)$ a diagram chase of the commutative diagram 
 \[
 \begin{tikzcd}
\,&  H^\ell(\underline{x}^j;M_1)\arrow{r}\arrow{d}{\alpha^\ell_{M_1;\underline{x};j;j+k}}  &  H^{\ell}(\underline{x}^j;M_2)\arrow{d}{\alpha^{\ell}_{M_2;\underline{x};j;j+k}} \\
  H^{\ell-1}(\underline{x}^{j+k};M_3)\arrow{r}\arrow{d}{\alpha^{\ell-1}_{M_3;\underline{x};j+k;j+k+1}}  &  H^\ell(\underline{x}^{j+k};M_1)\arrow{r}\arrow{d}{\alpha^\ell_{M_1;\underline{x};j+k;j+k+1}} &  H^{\ell}(\underline{x}^{j+k};M_2)\\
   H^{\ell-1}(\underline{x}^{j+k+1};M_3)\arrow{r} & H^\ell(\underline{x}^{j+k+1};M_1)
 \end{tikzcd}
 \]
 and knowing $\alpha^{\ell-1}_{M_3;\underline{x};j+k;j+k+1}$ is the $0$-map whenever $\ell-1\geq j$ is all that is needed.
 \end{proof}

\begin{proposition}\label{lcb and ses} Let $R$ be a commutative Noetherian ring, $0\to M_1\to M_2\to M_3\to 0$ a short exact sequence of $R$-modules, and $\underline{x}=x_1,\ldots,x_d$ a sequence of elements in $R$.
\begin{enumerate}
\item If $\underline{x}$ is a regular sequence on $M_1$ then $\lcb_i(\underline{x};M_2)=\lcb_{i}(\underline{x};M_3)$ for all $i\leq d-1$.
\item If $\underline{x}$ is a regular sequence on $M_2$ then $\lcb_i(\underline{x};M_3)=\lcb_{i+1}(\underline{x};M_1)$ for all $i\leq d-1$.
\item If $\underline{x}$ is a regular sequence on $M_3$ then $\lcb_i(\underline{x};M_1)=\lcb_{i}(\underline{x};M_2)$ for all $i\leq d$.
\end{enumerate}
\end{proposition}

\begin{proof} \textbf{Proof of (1)}: For $i<d$ we have $H^i(\underline{x}^j;M_1)=0$ and therefore if $i\leq d-2$ there are commutative diagrams
\[
\begin{tikzcd}
H^i(\underline{x}^j;M_2)\arrow{r}{\cong}\arrow{d}{\alpha^i_{M_2;\underline{x};j;j+k}} & H^i(\underline{x}^j;M_3)\arrow{d}{\alpha^i_{M_3;\underline{x};j;j+k}} \\
H^i(\underline{x}^{j+k};M_2)\arrow{r}{\cong} & H^{i}(\underline{x}^{j+k};M_3)
\end{tikzcd}
\]
whose horizontal arrows are isomorphisms. It readily follows that $\lcb_i(\underline{x};M_2)=\lcb_{i}(\underline{x};M_3)$ whenever $i\leq d -2$. Because $\underline{x}$ is a regular sequence on $M_1$ we have that the maps $\alpha^d_{M_1,\underline{x},j,j+k}$ are injective. Conside the following commutative diagrams whose rows are exact:
\[
\begin{tikzcd}
0 \arrow{r} & H^{d-1}(\underline{x}^j; M_2) \arrow{r}{\pi_j} \arrow{d}{\alpha^{d-1}_{M_2;\underline{x};j;j+k}} & H^{d-1}(\underline{x}^j; M_3)\arrow{r}{\delta_j} \arrow{d}{\alpha^{d-1}_{M_3;\underline{x};j;j+k}} & H^{d}(\underline{x}^j; M_1) \arrow{d}{\alpha^{d}_{M_1;\underline{x};j;j+k}} \\ 
0 \arrow{r} & H^{d-1}(\underline{x}^{j+k}; M_2)\arrow{r}{\pi_{j+k}} & H^{d-1}(\underline{x}^{j+k}; M_3) \arrow{r}{\delta_{j+k}} & H^{d}(\underline{x}^{j+k}; M_1)
\end{tikzcd}
\]
If  $\eta\in \ker(\alpha^{d-1}_{M_2;\underline{x};j,j+k})$ then $\pi_j(\eta)\in \ker(\alpha^{d-1}_{M_3;\underline{x};j,j+k})$. The maps $\pi_{j+k}$ are injective. Therefore $\alpha^{d-1}_{M_2;\underline{x};j,j+k}(\eta)=0$ whenever $k\geq \lcb_{d-1}(\underline{x};M_3)$ and hence $\lcb_{d-1}(\underline{x};M_2)\leq \lcb_{d-1}(\underline{x};M_3)$.

To show that $\lcb_{d-1}(\underline{x};M_2)\geq \lcb_{d-1}(\underline{x};M_3)$ consider an element $\eta\in \ker(\alpha^{d-1}_{M_3;\underline{x};j;j+k}).$ Then $\delta_j(\eta)\in  \ker(\alpha^{d}_{M_1;\underline{x};j;j+k}) $. But the maps $\alpha^{d}_{M_1;\underline{x};j;j+k}$ are injective and therefore $\delta_j(\eta)=0$. In particular, $\eta=\pi_j(\eta')$ for some $\eta'\in H^{d-1}(\underline{x}^j;M_2)$. The maps $\pi_{j+k}$ are all injective. Therefore $\eta'\in \ker(\alpha^{d-1}_{M_1;\underline{x};j;j+k})$ and it follows that $\alpha^{d-1}_{M_2;\underline{x};j;j+k}(\eta)=0$ whenever $k\geq \lcb_{d-1}(\underline{x};M_2)$. Therefore $\lcb_{d-1}(\underline{x};M_2)\geq \lcb_{d-1}(\underline{x};M_3)$ and hence $\lcb_{d-1}(\underline{x};M_2)= \lcb_{d-1}(\underline{x};M_3)$. This completes the proof of $(1)$.

\noindent\textbf{Proof of (2)}: Because we are assuming that $\underline{x}$ is a regular sequence on $M_2$ it follows that $H^i(\underline{x}^j;M_2)=0$ whenever $i\leq d-1$ and therefore if $i\leq d-2$ there are commutative diagrams
\[
\begin{tikzcd}
H^i(\underline{x}^j;M_3)\arrow{r}{\cong}\arrow{d}{\alpha^{i+1}_{M_3;\underline{x};j;j+k}} & H^{i+1}(\underline{x}^j;M_1)\arrow{d}{\alpha^i_{M_1;\underline{x};j;j+k}} \\
H^i(\underline{x}^{j+k};M_3)\arrow{r}{\cong} & H^{i+1}(\underline{x}^{j+k};M_1)
\end{tikzcd}
\]
whose horizontal arrows are isomorphisms. It easily follows that $\lcb_i(\underline{x};M_3)=\lcb_{i+1}(\underline{x}; M_1)$ whenever $i\leq d-2$. To verify that $\lcb_{d-1}(\underline{x};M_3)=\lcb_{d}(\underline{x}; M_1)$ consider the following commutative diagrams:
\[
\begin{tikzcd}
0 \arrow{r} & H^{d-1}(\underline{x}^j; M_3) \arrow{r}{\delta_j} \arrow{d}{\alpha^{d-1}_{M_3;\underline{x};j;j+k}} & H^{d}(\underline{x}^j; M_1)\arrow{r}{i_j} \arrow{d}{\alpha^{d}_{M_1;\underline{x};j;j+k}} & H^{d}(\underline{x}^j; M_2) \arrow{d}{\alpha^{d}_{M_2;\underline{x};j;j+k}} \\ 
0 \arrow{r} & H^{d-1}(\underline{x}^{j+k}; M_3)\arrow{r}{\delta_{j+k}} & H^{d}(\underline{x}^{j+k}; M_1) \arrow{r}{i_{j+k}} & H^{d}(\underline{x}^{j+k}; M_2)
\end{tikzcd}
\]
Similar to the proof of $(1)$, a simple diagram chase and utilizing the injectivity of the maps $\delta_j,\delta_{j+k},$ and $\alpha^d_{M_2;\underline{x};j,j+k}$ will imply $\lcb_{d-1}(\underline{x};M_3)=\lcb_{d}(\underline{x}; M_1)$.

\noindent\textbf{Proof of (3)}: Similar to the proofs of $(1)$ and $(2)$, if $i\leq d-1$ there are commutative squares
\[
\begin{tikzcd}
H^i(\underline{x}^j;M_1)\arrow{r}{\cong}\arrow{d}{\alpha^i_{M_1;\underline{x};j;j+k}} & H^i(\underline{x}^j;M_2)\arrow{d}{\alpha^i_{M_2;\underline{x};j;j+k}} \\
H^i(\underline{x}^{j+k};M_1)\arrow{r}{\cong} & H^i(\underline{x}^j;M_2)
\end{tikzcd}
\]
whose horizontal arrows are isomorphisms. There will also be commutative diagrams
\[
\begin{tikzcd}
0 \arrow{r} & H^{d}(\underline{x}^j; M_1) \arrow{r}{i_j} \arrow{d}{\alpha^{d}_{M_1;\underline{x};j;j+k}} & H^{d}(\underline{x}^j; M_2)\arrow{r}{\pi_j} \arrow{d}{\alpha^{d}_{M_2;\underline{x};j;j+k}} & H^{d}(\underline{x}^j; M_3) \arrow{d}{\alpha^{d}_{M_3;\underline{x};j;j+k}} \\ 
0 \arrow{r} & H^{d}(\underline{x}^{j+k}; M_1)\arrow{r}{i_{j+k}} & H^{d}(\underline{x}^{j+k}; M_2) \arrow{r}{\pi_{j+k}} & H^{d}(\underline{x}^{j+k}; M_3).
\end{tikzcd}
\]
Utilizing the commutative square above will show $\lcb_i(\underline{x};M_1)=\lcb_i(\underline{x};M_2)$ whenever $i\leq d-1$. A simple diagram chase of the second diagram and utilizing the injectivity of the maps $i_j,i_{j+k}$, and $\alpha^d_{M_3;\underline{x};j;j+k}$ imply $\lcb_d(\underline{x};M_1)=\lcb_d(\underline{x};M_2)$.
\end{proof}

\section{Equality of test ideals}\label{section weak implies strong method}

The proof of Theorem~\ref{Main tight closure theorem} goes as follows: Theorem~\ref{theorem how to make test ideals agree using lcbs}, in combination with Lemma~\ref{Lemma Equiv way to check weak and strong}, shows that the test ideals of a local ring agree provided there exists suitable parameters which enjoy prescribed local cohomology bounds. Proposition~\ref{proposition technical lcb} and Proposition~\ref{proposition on how to make test ideals agree} can then be combined to show that a parameter sequence satisfies the hypotheses of Theorem~\ref{theorem how to make test ideals agree using lcbs} provided that the parameter sequence annihilates a family of $\Ext$-modules in a controlled way. Theorem~\ref{theorem when are technical condtions met} provides to us a suitable system of parameters so that the desired annihilation properties of the previous propositions are met under the assumptions of Theorem~\ref{Main tight closure theorem}

\subsection{Sufficient conditions which imply equality of test ideals}

The content of the following lemma can be pieced together by work of the first author in  \cite{Aberbach2002}. We refer the reader to \cite[Lemma~6.7]{PolstraTucker} for a direct presentation of the lemma.\footnote{In \cite[Lemma~6.7]{PolstraTucker} there is an assumption that $R$ is complete. But observe that since $R\to \widehat{R}$ is faithfully flat the claims of the lemma can be checked after completion.}

\begin{lemma}
\label{colonslemma}
Suppose that $(R,\fm,k)$ is a Cohen-Macaulay local normal domain of dimension $d$, and $J\subseteq R$ an ideal of pure height $1$.  Let $x_1, \ldots, x_d \in R$ be a suitable system of parameters for $R$ with respect to $J$, and fix $e \in \N$.
\begin{enumerate}
\item
If $x_2 J \subseteq a_2 R$ for some $a_2 \in J$, then for any non-negative integers $N_2, \ldots, N_d$ with $N_2 \geq 2$, we have that
\[
\begin{array}{ll}
 & ((J^{(p^{e})},x_2^{N_2 p^e}, x_3^{N_3 p^e}, \ldots, x_d^{N_d p^e}):x_2^{(N_2 - 1)p^e} ) \\ = &
 ((J^{[p^{e}]},x_2^{N_2 p^e}, x_3^{N_3 p^e}, \ldots, x_d^{N_d p^e}):x_2^{(N_2 - 1)p^e} ) \\ =  & ((J^{[p^{e}]},x_2^{2 p^e}, x_3^{N_3 p^e}, \ldots, x_d^{N_d p^e}):x_2^{p^e} ).
\end{array}
\]
\item
Suppose $x_d^n J^{(m)} \subseteq a_d R\subseteq J^{(m)}$, then for any non-negative integers $N_2, \ldots, N_d$ with $N_d \geq 2$, we have that
\begin{equation*}
\begin{array}{ll}
& ((J^{(p^{e})},x_2^{N_2 p^e},\ldots, x_{d-1}^{N_{d-1} p^e},  x_d^{N_d p^e}):  x_d^{(N_d -1) p^e}) \\  \subseteq 
& ((J^{(p^{e})},x_2^{N_2 p^e},\ldots, x_{d-1}^{N_{d-1} p^e},  x_d^{2 p^e}): x_1^{m} x_d^{p^e}).
\end{array}
\end{equation*}

\end{enumerate}
\end{lemma}

\begin{theorem}
\label{theorem how to make test ideals agree using lcbs}
Let $(R,\fm,k)$ be a local normal Cohen-Maculay domain of Krull dimension $d$, $\Q$-Gorenstein in codimension $2$, and of prime characteristic $p>0$. Assume that $R$ has a test element.  Let $J_1\subseteq R$ be a choice of canonical ideal and $m\in \N$ such that $J_1^{(m)}$ is principal in codimension $2$. Suppose $x_1,\ldots,x_d$ is a suitable system of parameters with respect to $J_1$ such that the following conditions are met:
\begin{itemize}
\item There exists element $a_2\in J_1$ such that $x_2J_1\subseteq a_2 R$ and there exists element $a_3\in J_1^{(m)}$ such that $x_3J_1^{(m)}\subseteq a_3R$;
\item For each $i\in \N$ there exists an integer $\ell$ such that
\[
\lcb_{d-1}(x^{\ell}_2,x^{\ell }_3,x_4,\ldots,x_d; R/J_1^{(mi+1)})\leq i+1.
\]
\end{itemize}
Then $0^*_{H^{d-1}_\fm(R/J_1)}=0^{*,fg}_{H^{d-1}_\fm(R/J_1)}$. 
\end{theorem}

\begin{proof}
 Identify $H^{d-1}_{\fm}(R/J_1)$ as 
\[
\varinjlim_{s} \left( H^{d-1}(x_2^s,x_3^s,\ldots,x_d^{s};R/J_1)\xrightarrow{\cdot (x_2\cdots x_d)}H^{d-1}(x_2^{s+1},x_3^{s+1},\ldots, x_d^{s+1};R/J_1)\right).
\]
Suppose that $\eta\in 0^*_{H^{d-1}_{\fm}(R/J_1)}$ and $\eta=r+(J_1,x^t_2,x^t_3,\ldots,x^t_d)$. Let $y_2,\ldots,y_d$ denote the parameter sequence $x^t_2,x^t_3,\ldots,x^t_d$ and identify $H^{d-1}_\fm(R/J_1)$ as 
\[
\varinjlim_{s} \left( H^{d-1}(y_2^s,y_3^s,\ldots,y_d^{s};R/J_1)\xrightarrow{\cdot (y_2\cdots y_d)}H^{d-1}(y_2^{s+1},y_3^{s+1},\ldots, y_d^{s+1};R/J_1)\right).
\]
In particular, $F^e_R(H^{d-1}_{\fm}(R/J_1)) $ is isomorphic to the following direct limit:
\[
 \varinjlim_{s} \left( H^{d-1}\left(y_2^{sp^e},y_3^{sp^e},\ldots, y_d^{sp^e};\frac{R}{J_1^{[p^e]}}\right)\xrightarrow{\cdot (y_2\cdots y_d)^{p^e}}H^{d-1}\left(y_2^{(s+1)p^e},y_3^{(s+1)p^e},\ldots, y_d^{(s+1)p^e};\frac{R}{J_1^{[p^e]}}\right)\right).
\]
Before continuing, we point out that the sequence of elements $y_2,\ldots,y_d$ satisfy the hypotheses of the theorem, see Lemma~\ref{lemma lcb power of elements}.

Let $c\in R^\circ$ be a test element. Because we are assuming $\eta\in 0^*_{H^{d-1}_{\fm}(R/J_1)}$ we have that for each integer $e\in\mathbb{N}$ there exists an integer $s$ so that 
\begin{align*}
cr^{p^e}(y_2y_3\cdots y_d)^{(s-1)p^e}&\in (J_1,y_2^s,y_3^s,\ldots,y_d^s)^{[p^e]}\\
&\subseteq (J_1^{(p^e)},y_2^{sp^e},y_3^{sp^e},\ldots,y_d^{sp^e}).
\end{align*}
Observe that $p^e\geq m(\lfloor\frac{p^e}{m}\rfloor-1)+1$ and hence $J_1^{[p^e]}\subseteq J_1^{(p^e)}\subseteq J_1^{(m(\lfloor\frac{p^e}{m}\rfloor-1)+1)}$. If we let $i=(\lfloor\frac{p^e}{m}\rfloor-1)$ then 
\[
cr^{p^e}(y_2y_3\cdots y_d)^{(s-1)p^e}\in (J_1^{(mi+1)},y_2^{sp^e},y_3^{sp^e},\ldots,y_d^{sp^e}).
\]
Let $\ell$ be a choice of integer depending on $i$ as in the statement of the theorem. If we multiply the above containment by $(y_2y_3)^{(\ell-1)sp^e}$ we find that
\begin{align*}
cr^{p^e}(y_2y_3\cdots y_d)^{(s-1)p^e}\cdot (y_2y_3)^{(\ell-1)sp^e} &= cr^{p^e}(y_2y_3)^{(\ell-1)p^e}(y_2^\ell y_3^\ell  y_4 \cdots y_d)^{(s-1)p^e}\\
&\in (J_1^{(mi+1)},y_2^{\ell s p^e},y_3^{\ell s p^e},y_4^{sp^e},\cdots, y_d^{sp^e}).
\end{align*}
Now consider the element 
\[
\zeta = cr^{p^e}(y_2y_3)^{(\ell-1)p^e}+ (y_2^{\ell p^e},y_3^{\ell p^e},y_4^{p^e},\ldots, y_d^{p^e})
\]
of the Koszul cohomology group 
\[
H^{d-1}(y_2^{\ell p^e},y_3^{\ell p^e}, y_4^{p^e},\ldots, y_d^{p^e};R/J_1^{(mi+1)})\cong R/(J_1^{(mi+1)},y_2^{\ell p^e},y_3^{\ell p^e}, y_4^{p^e},\ldots, y_d^{p^e}).
\]
We have shown that multiplying $cr^{p^e}(y_2y_3)^{(\ell-1)p^e}$ by $(y_2^\ell y_3^\ell  y_4 \cdots y_d)^{(s-1)p^e}$ gives an element of the ideal $(J_1^{(mi+1)},y_2^{\ell s p^e},y_3^{\ell s p^e},y_4^{sp^e},\cdots, y_d^{sp^e})$. Equivalently,
\begin{align*}
\alpha_{R/J_1^{(mi+1)};y_2^\ell,y_3^\ell, y_4,\ldots,y_d;p^e;sp^e}&(\zeta)\\
&=cr^{p^e}(y_2y_3)^{(\ell-1)p^e}(y_2^\ell y_3^\ell  y_4 \cdots y_d)^{(s-1)p^e}+(y_2^{\ell s p^e},y_3^{\ell s p^e},y_4^{sp^e},\cdots, y_d^{sp^e})
\end{align*}
is the $0$-element of
\[
H^{d-1}(y_2^{\ell sp^e},y_3^{\ell sp^e}, y_4^{sp^e},\ldots, y_d^{sp^e};R/J_1^{(mi+1)})\cong R/(J_1^{(mi+1)},y_2^{\ell sp^e},y_3^{\ell sp^e}, y_4^{sp^e},\ldots, y_d^{sp^e}).
\]
We are assuming 
\[
\lcb_{d-1}(x^{\ell}_2,x^{\ell }_3,x_4,\ldots,x_d; R/J_1^{(mi+1)})\leq i+1.
\]
and by Lemma~\ref{lemma lcb power of elements}
\[
\lcb_{d-1}(y^{\ell}_2,y^{\ell }_3,y_4,\ldots,y_d; R/J_1^{(mi+1)})\leq i+1\leq p^e-1+1\leq p^e.
\]
Therefore
\begin{align*}
\alpha_{R/J_1^{(mi+1)};y_2^\ell,y_3^\ell, y_4,\ldots,y_d;p^e;p^e+p^e}(\zeta)&=\alpha_{R/J_1^{(mi+1)};y_2^\ell,y_3^\ell, y_4,\ldots,y_d;p^e;2p^e}(\zeta)\\
&=cr^{p^e}(y_2y_3)^{(\ell-1)p^e}(y_2^\ell y_3^\ell  y_4 \cdots y_d)^{p^e}+(y_2^{\ell 2 p^e},y_3^{\ell 2 p^e},y_4^{2p^e},\cdots, y_d^{2p^e})
\end{align*}
is the $0$-element of
\[
H^{d-1}(y_2^{\ell 2p^e},y_3^{\ell 2p^e}, y_4^{2p^e},\ldots, y_d^{2p^e};R/J_1^{(mi+1)})\cong R/(J_1^{(mi+1)},y_2^{\ell 2p^e},y_3^{\ell 2p^e}, y_4^{2p^e},\ldots, y_d^{2p^e}).
\]
Equivalently, the element 
\[
cr^{p^e}(y_2y_3)^{(\ell-1)p^e}(y_2^\ell y_3^\ell  y_4 \cdots y_d)^{p^e}=c(ry_4\cdots y_d)^{p^e} (y_2y_3)^{(2\ell-1)p^e}
\]
is an element of the ideal
\[
(J_1^{(mi+1)},y_2^{2\ell p^e},y_3^{2\ell p^e},y_4^{2p^e},\cdots, y_d^{2p^e}).
\]
Recall that $i=\lfloor \frac{p^e}{m}\rfloor-1$ and observe that $m\lfloor \frac{p^e}{m}\rfloor\geq m(\frac{p^e}{m}-1)=p^e-m$. Hence $mi+1=m(\lfloor \frac{p^e}{m}\rfloor-1)+1\geq p^e-(2m-1)$ and
\[
c(ry_4\cdots y_d)^{p^e} (y_2y_3)^{(2\ell-1)p^e}\in (J_1^{(p^e-(2m-1))},y_2^{2\ell p^e},y_3^{2\ell p^e},y_4^{2p^e},\cdots, y_d^{2p^e}).
\]
Multiplying by $x_1^{2m-1}$ we find that
\[
x_1^{2m-1}c(ry_4\cdots y_d)^{p^e} (y_2y_3)^{(2\ell-1)p^e}\in (J_1^{(p^e)},y_2^{2\ell p^e},y_3^{2\ell p^e},y_4^{2p^e},\cdots, y_d^{2p^e}).
\]
If we apply Lemma~\ref{colonslemma}(2) with respect to the element $y_3$ we find that
\[
x_1^{3m-1}c(ry_4\cdots y_d)^{p^e} y_2^{(2\ell-1)p^e}y_3^{p^e}\in (J_1^{(p^e)},y_2^{2\ell p^e},y_3^{2 p^e},y_4^{2p^e},\cdots, y_d^{2p^e}).
\]
Next we apply Lemma~\ref{colonslemma}(1) with respect to the element $y_2$ and find that 
\[
x_1^{3m-1}c(ry_4\cdots y_d)^{p^e} y_2^{p^e}y_3^{p^e}=x_1^{3m-1}c(ry_2y_3y_4\cdots y_d)^{p^e} \in (J_1^{[p^e]},y_2^{2p^e},y_3^{2 p^e},y_4^{2p^e},\cdots, y_d^{2p^e}).
\]
The element $x_1^{3m-1}c$ does not depend on $e$ and therefore 
\[
ry_2y_3y_4\cdots y_d \in (J_1,y^2_2,y_3^{2},y_4^{2},\cdots, y_d^{2})^*.
\]
Therefore 
\[
\eta = r+(y_2,\ldots,y_d)=ry_2y_3y_4\cdots y_d+(J_1,y^2_2,y_3^{2},y_4^{2},\cdots, y_d^{2})
\]
is an element of $0^{*,fg}_{H^{d-1}_\fm(R/J_1)}$.
\end{proof}

The next two propositions provide the linear bound of top local cohomology bounds of the family of $R$-modules $\left\{R/J_1^{(mi+1)}\right\}$ described in Theorem~\ref{theorem how to make test ideals agree using lcbs} whenever there exists a suitable system of parameters which annihilates a family of $\Ext$-modules in a controlled manner.

\begin{proposition}\label{proposition technical lcb} Let $(R,\fm,k)$ be a local normal Cohen-Macaulay domain of Krull dimension $d$ and $\Q$-Gorenstein in codimension $2$. Assume that $R$ has a test element.  Let $J_1\subseteq R$ be a choice of canonical ideal and $m\in \N$ such that $J_1^{(m)}$ is principal in codimension $2$. Suppose $S$ is a regular local ring mapping onto $R$, $R\cong S/P$, and $\Ht(P)=h$. Let $x_1,\ldots,x_d$ be a suitable system of parameters with respect to $J_1$ such that for each integer $i\in \N$ and $2\leq j\leq d-2$
\[
(x_2^{i},x_3^i,\ldots ,x_{j+2}^i)\Ext^{h+j}_S(\Ext^{h+1}_S(R/J_1^{mi+1},S),S)=0.
\]
Then for each integer $i\in\N$
\[
\lcb_{d-1}(x^{d-1}_2,x^{d-1}_3,\ldots,x^{d-1}_d;\Ext^{h+1}_S(\Ext^{h+1}_S(R/J_1^{mi+1},S),S))\leq i.
\]
\end{proposition}

\begin{proof}
Let $J_i=J_1^{mi+1}$ and let $(F_\bullet,\partial_\bullet)$ be the minimal free $S$-resolution of $\Ext^{h+1}_S(R/J_i,S)$. Denote by $(-)^*$ the functor $\Hom_S(-,S)$ and consider the dualized complex $(F_\bullet^*,\partial_\bullet^*)$. For every $j\geq 1$ there are short exact sequences
\[
0\to \Ext^{h+j}_S(\Ext^{h+1}_S(R/J_i,S),S)\to \coker(\partial_{h+j}^*)\to \im(\partial_{h+j+1}^*)\to 0
\]
and
\[
0\to \im(\partial_{h+j+1}^*)\to F_{h+j+1}^*\to \coker(\partial_{h+j+1}^*)\to 0.
\]
The $S$-module $\coker(\partial_{h+1}^*)$ has projective dimension $h+1$ and the height $h+1$ ideal $J_i$ annihilates the submodule $\Ext^{h+1}_S(\Ext^{h+1}_S(R/J_i,S),S)$. By a simple prime avoidance argument we may lift $\underline{x}=x_2,\ldots,x_d$ to elements of $S$ and assume that $\underline{x}$ is a regular sequence on $\coker(\partial_{h+1}^*)$ and the free $S$-modules $F_i^*$.

The module $\Ext^{h+1}_S(R/J_1,S)$ is an $(S_2)$-module over its support, see Lemma~\ref{lemma s2-module}. In particular, 
\[
\Ext^{h+d}_S(\Ext^{h+1}_S(R/J_i,S),S)=\Ext^{h+d-1}_S(\Ext^{h+1}_S(R/J_i,S),S)=0
\]
and 
\[
\coker(\partial^*_{h+d-2})\cong \Ext^{h+d-2}_S(\Ext^{h+1}_S(R/J_i,S),S).
\]
Consider the short exact sequence
\[
0\to \im(\partial_{h+d-2}^*)\to F_{h+d-2}^*\to   \Ext^{h+d-2}_S(\Ext^{h+1}_S(R/J_i,S),S)\to 0.
\]
We are assuming $(x_2^i,x_3^i,\ldots,x_d^i) \Ext^{h+d-2}_S(\Ext^{h+1}_S(R/J_i,S),S)=0$ for every $i\in \N$. By $(2)$ of Proposition~\ref{lcb and ses} and $(3)$ of Proposition~\ref{second proposition 0 map Koszul cohomology} we have that $\lcb_2(\underline{x}; \im(\partial_{h+d-2}^*))\leq i$.
Next, we consider the short exact sequence
\[
0\to \Ext^{h+d-3}_S(\Ext^{h+1}_S(R/J_i,S),S)\to \coker(\partial_{h+d-3}^*)\to \im(\partial_{h+d-2}^*)\to 0.
\]
We established $\lcb_2(\underline{x}; \im(\partial_{h+d-2}^*))\leq i$ and we are assuming 
\[
(x_2^i,\ldots,x_{d-1}^i)\Ext^{h+d-3}_S(\Ext^{h+1}_S(R/J_i,S),S)=0
\] 
for every $i\in \N$. By (1) of Proposition~\ref{second proposition 0 map Koszul cohomology} we have
\[
\lcb_2(\underline{x};\coker(\partial_{h+d-3}^*))\leq i+i=2i.
\]
Next consider the short exact sequence
\[
0\to \im(\partial^*_{h+d-3})\to F_{h+d-3}^*\to \coker(\partial_{h+d-3}^*)\to 0.
\]
By $(2)$ of Proposition~\ref{lcb and ses} and knowing that $\lcb_2(\underline{x};\coker(\partial_{h+d-3}^*))\leq 2i$ we see that
\[
\lcb_3(\underline{x};\im(\partial^*_{h+d-3}))\leq 2i.
\]
Inductively, we find that 
\[
\lcb_j(\underline{x};\im(\partial^*_{h+d-j}))\leq (j-1)i
\]
and
\[
\lcb_j(\underline{x};\coker(\partial_{h+d-j-1}^*))\leq ji
\]
for each $2\leq j\leq d-1$. In particular,
\[
\lcb_{d-1}(\underline{x};\Ext^{h+1}_S(\Ext^{h+1}_S(R/J_i,S),S))\leq (d-1)i.
\]
By Lemma~\ref{lemma lcb power of elements} the parameter sequence $\underline{x}^{d-1}=x_2^{d-1},\ldots, x_d^{d-1}$ on $R/J_1$ satisfies 
\[
\lcb_{d-1}(\underline{x}^{d-1};\Ext^{h+1}_S(\Ext^{h+1}_S(R/J_i,S),S)))\leq i
\]
for each integer $i\in \N$.
\end{proof}

\begin{proposition}\label{proposition on how to make test ideals agree} Let $(R,\fm,k)$ be a local normal Cohen-Maculay domain of Krull dimension $d\geq 4$ which is $\Q$-Gorenstein in codimension $2$.  Let $J_1\subseteq R$ be a choice of canonical ideal and $m\in \N$ such that $J_1^{(m)}$ is principal in codimension $2$. Suppose $S$ is a regular local ring mapping onto $R$, $R\cong S/P$, and $\Ht(P)=h$. Let $x_1,\ldots,x_d$ be a suitable system of parameters with respect to $J_1$ such that:

\begin{itemize}
\item The ideals $J^{(m)}_1R_{x_2}$ and $J^{(m)}_1R_{x_3}$ are principal ideals in their respective localizations;
\item For each integer $i\in \N$ and $2\leq j\leq d-2$
\[
(x_2^{i},x_3^i,\ldots ,x_{j+2}^i)\Ext^{h+j}_S(\Ext^{h+1}_S(R/J^{mi+1},S),S)=0.
\]
Then the following hold:
\end{itemize}
\begin{enumerate}
\item For each integer $i\in \N$ there exists an integer $\ell$ such that
\[
\lcb_{d-1}(x^{\ell(d-1)}_2,x^{\ell(d-1)}_3,x^{d-1}_4,\ldots,x^{d-1}_d; R/J^{(mi+1)})\leq i+1;
\]
\item For each integer $i\in \N$ there exists an integer $\ell$ such that
\[
\lcb_{d-1}(x^{\ell(d-1)}_2,x^{\ell(d-1)}_3,x^{d-1}_4,\ldots,x^{d-1}_d; R/J^{mi+1})\leq i+2.
\]
\end{enumerate}
\end{proposition}

\begin{proof}
For each $i\in \N$ let $C_i$ be the cokernel of
\[
R/J^{mi+1}\to \Ext^{h+1}_S(\Ext^{h+1}_S(R/J_1^{mi+1},S),S)
\]
and consider the short exact sequences 
\[
0\to R/J_1^{(mi+1)}\to \Ext^{h+1}_S(\Ext^{h+1}_S(R/J_1^{mi+1},S),S)\to C_i\to 0,
\]
see Lemma~\ref{Inclusion lemma} for details.

By Lemma~\ref{Inclusion lemma} the module $C_i$ is $0$ when either $x_2$ or $x_3$ is inverted. Hence for each $i\in \N$ there exists an integer $\ell$ such that $(x_2^\ell,x_3^\ell)C_i=0$. Because $d\geq 4$ we have that $d-1\geq 3$ and $(3)$ of Proposition~\ref{second proposition 0 map Koszul cohomology} is applicable and implies
\begin{align*}
\lcb_{d-1}(x_2^\ell,x_3^\ell,x_4,\ldots,x_d; & R/J^{(mi+1)})\leq\\
& \lcb_{d-1}(x_2^\ell,x_3^\ell,x_4,\ldots,x_d;  \Ext^{h+1}_S(\Ext^{h+1}_S(R/J^{mi+1},S,S))+1.
\end{align*}
Statement (1) follows by Proposition~\ref{proposition technical lcb}.

To prove $(2)$ let $K_i=J_1^{(mi+1)}/J_1^{mi+1}$ and consider the short exact sequences
\[
0\to K_i\to R/J_1^{mi+1}\to R/J_1^{(mi+1)}\to 0.
\]
The module $K_i$ is $0$ when either $x_2$ or $x_3$ are inverted. Hence for each $i\in \N$ there exists an integer $\ell$ such that $(x_2^\ell,x_3^\ell)K_i=0$. By (1) of Proposition~\ref{second proposition 0 map Koszul cohomology} we have that
\[
\lcb_{d-1}(x_2^\ell,x_3^\ell,x_4,\ldots,x_d;  R/J^{mi+1})\leq
 \lcb_{d-1}(x_2^\ell,x_3^\ell,x_4,\ldots,x_d; R/J^{(mi+1)})+1\leq i+2.
\] 
\end{proof}

\subsection{Main results} We have arrived at the main theorem of the article. Theorem~\ref{Main theorem dimension 4} and Theorem~\ref{Main tight closure theorem} are consequences of the next theorem. Theorem~\ref{theorem when are technical condtions met} below gives the existence of suitable system of parameters satisfying the annihilation properties of Proposition~\ref{proposition technical lcb} and Proposition~\ref{proposition on how to make test ideals agree} whenever an anti-canonical ideal has analytic spread at most $2$ and reduction number $1$ on the punctured spectrum.

\begin{theorem}\label{theorem when are technical condtions met} Let $(R,\fm,k)$ be an excellent local normal Cohen-Macaulay domain of Krull dimension $d\geq 4$ which is $\Q$-Gorenstein in codimension $2$.  Let $J_1\subsetneq R$ be a choice of a canonical ideal and $x_1\in J_1$ a generic generator of $J_1$. Suppose $(x_1)=J_1\cap K_1$ so that $K_1$ is an anti-canonical ideal of $R$. Suppose further that there exists integer $m'$ such that $K^{(m')}_1$ has analytic spread at most $2$ and reduction number $1$ with respect to some reduction on the punctured spectrum.  Then there exists an integer $m\in \N$ and suitable parameters $x_2,\ldots ,x_d$ on $R/J_1$ such that
\[
(x^i_2,x^i_3,\ldots ,x^i_{j+2})\Ext^{h+j}_S(\Ext^{h+1}_S(R/J_1^{mi+1},S),S)=0
\]
for every integer $i \in \N$ and $2\leq j\leq d-2$.
\end{theorem}

\begin{proof}
We can choose $m''\in \N$ so that $K=K_1^{(m'')}$ is principal in codimension $2$. If $m$ is any multiple of $m'$ and $m''$ then $K_1^{(m)}$ is principal in codimension $2$ and has analytic spread at most $2$ on $\Spec(R)-\{\fm\}$, see Proposition~\ref{Proposition Noetherian if analytic spread 2} to know that the symbolic Rees ring $\mathcal{R}_{K_1^{(m')}}$ is Noetherian on the punctured spectrum and Proposition~\ref{Proposition analytic spread of symbolic powers} to insure that the analytic spread of $K^{(m)}$ is no more than the analytic spread of $K_1^{(m')}$ on the punctured spectrum. By Proposition~\ref{Proposition finite generation on punctured spectrum} we can choose $m$ to be a multiple of $m'$ and $m''$ and such that the containment of ideals $K^i\subseteq K^{(i)}$ is an equality on $\Spec(R)-\{\fm\}$ for each $i\in \N$. Let $K=K_1^{(m)}$ and let $x=x_1^m$.

\begin{claim} For each integer $i\in \N$
\[
\Ext^{h+1}_S(R/J_1^{mi+1},S)\cong x_1K_1^{(mi)}/x_1^{mi+1}J_1\cong K_1^{(mi)}/x_1^{mi}J_1=K^{(i)}/x^iJ_1.
\]
\end{claim}

\begin{proof}[Proof of claim] For each integer $i$ consider the following short exact sequence:
\[
0\to \frac{J_1^{mi+1}}{x_1^{mi+1}J_1}\to \frac{R}{x_1^{mi+1}J_1}\to \frac{R}{J_1^{mi+1}}\to 0
\]
The ideal $x_1^{mi+1}J_1$ is isomorphic to the canonical module of $R$, therefore 
\[
\Ext^{h+1}_S(R/x_1^{mi+1}J_1,S)\cong R/x_1^{mi+1}J_1,
\] 
and there are exact sequences
\[
0\to \Ext^{h+1}_S(R/J^{mi+1},S)\to \frac{R}{x_1^{mi+1}J_1}\to \Ext^{h+1}_S(J_1^{mi+1}/x_1^{mi+1}J_1,S).
\]
Therefore $\Ext^{h+1}_S(R/J_1^{mi+1},S)\cong L_i/x_1^{mi+1}J_1$ for some ideal $L_i\subseteq R$. Moreover, $R/L_i\subseteq  \Ext^{h+1}_S(J_1^{mi+1}/x_1^{mi+1}J_1,S)$. Because $\Ext^{h+1}_S(J_1^{mi+1}/x_1^{mi+1}J_1,S)$ is an $(S_2)$-module over its support, see Lemma~\ref{lemma s2-module}, it follows that $R/L_i$ is an $(S_1)$-module over its support. Hence $L_i$, as an ideal of $R$, is unmixed of height $1$. Moreover, every component of $L_i$ is a component of $x_1 R$. Localizing at a component of $J_1$ we see that $L_i$ agrees with $x_1 R$ and localizing at a component of $K_1$ we see that $L_i$ agrees with $x_1^{mi+1}$. Therefore $L_i$ agrees with the unmixed ideal $x_1K_1^{(mi)}$ and so
\[
\Ext^{h+1}_S(R/J_1^{mi+1},S)\cong x_1K_1^{(mi)}/x_1^{mi+1}J_1.
\]
The second isomorphism 
\[
x_1K_1^{(mi)}/x_1^{mi+1}J_1\cong K_1^{(mi)}/x_1^{mi}J_1=K^{(i)}/x^iJ_1
\]
is division by $x_1$.
\end{proof}

\begin{claim}\label{claim on ext-isoms} For all integers $i,j\in \N$ and $j\geq 2$
\[
\Ext^{h+j}_S(K^{(i)}/x^iJ_1,S)\cong \Ext^{h+j+1}_S(R/K^{(i)},S).
\]
\end{claim}
\begin{proof}[Proof of claim] For each integer $i\in \N$ consider the short exact sequence
\[
0\to K^{(i)}/x^iJ_1\to R/x^iJ_1\to R/K^{(i)}\to 0.
\]
The cyclic $R$-module $R/x^iJ_1$ is Cohen-Macaulay of dimension $d-1$ and therefore 
\[
\Ext^{h+j}_S(R/x^iJ_1,S)=0
\]
for all $j\geq 2$ and hence $\Ext^{h+j}_S(K^{(i)}/x^iJ_1,S)\cong \Ext^{h+j+1}_S(R/K^{(i)},S)$.
\end{proof}
To prove the theorem it is now enough to find parameters $x_2,x_3\ldots, x_{d}$ on $R/J_1$ such that
\[
(x^i_2,x^i_3,\ldots,x^i_{j+2})\Ext^{h+j+1}_S(R/K^{(i)},S)
\]
for every integer $i\in \N$ and $2\leq j\leq d-2$.

\begin{claim} For every integer $i\in \N$ and $2\leq j\leq d-2$
\[
\Ext^{h+j+1}_S(R/K^{(i)},S)\cong \Ext^{h+j+1}_S(R/K^{i},S).
\]
\end{claim}

\begin{proof}[Proof of claim] Consider the short exact sequences
\[
0\to \frac{K^{(i)}}{K^i}\to \frac{R}{K^i}\to \frac{R}{K^{(i)}}\to 0.
\]
For each $i\in \N$ the modules $K^{(i)}/K^i$ are supported only at the maximal ideal.  In particular, $\Ext^{\ell}_S(K^{(i)}/K^i,S)=0$ for all $\ell\leq d+h-1$ and the claim follows.
\end{proof}
To prove the theorem it is now enough to find parameters $x_2,x_3\ldots, x_{d}$ on $R/J_1$ such that
\[
(x^i_2,x^i_3,\ldots,x^i_{j+2})\Ext^{h+j+1}_S(R/K^{i},S)
\]
for every integer $i\in \N$ and $2\leq j\leq d-2$.

We can choose parameters $x_2$ and $x_3$ on $R/J_1$ such that $KR_{x_2}$ and $KR_{x_3}$ are principal ideals in their respective localizations. Suppose $x_2$ has been chosen such that $KR_{x_2}=(a)R_{x_2}$, $a\in K$, and $x_2K\subseteq (a)$. Then $x_2^iK^i\subseteq (a^i)$ and therefore the left term of the following short exact sequence is annihilated by $x_2^i$:
\[
0\to \frac{K^i}{(a^i)}\to \frac{R}{(a^i)}\to \frac{R}{K^i}\to 0.
\]
It follows that $\Ext_S^{h+j+1}(R/K^i,S)$ is isomorphic to $\Ext_S^{h+j}(K^i/(a^i),S)$ for every $2\leq j\leq d-2$ and therefore
\[
x_2^i\Ext_S^{h+j+1}(R/K^i,S)=0
\]
for every $i$ and $2\leq j\leq d-2$.

Similarly, we can find $x_3$ a parameter on $R/(J_1,x_2)$ such that
\[
x_3^i\Ext_S^{h+j+1}(R/K^i,S)=0
\]
for every $i$ and $2\leq j\leq d-2$.

Assume we have found parameters $x_2,x_3,\ldots,x_\ell$ on $R/J_1$ such that
\[
x_m^i\Ext^{h+j+1}_S(R/K^i,S)=0
\]
for every $2\leq m\leq \ell$, $i\geq 1$, and $m-2\leq j\leq d-2$. We wish to find parameter element $x_{\ell+1}$ of $R/(J_1,x_2,\ldots,x_{\ell})$ such that
\[
x_{\ell+1}^i\Ext^{h+j+1}_S(R/K^i,S)=0
\]
for every $i\in \N$ and $\ell-1\leq j\leq d-2$.

\begin{claim}\label{analytic spread claim}  Let $\Lambda=\{P_1,\ldots,P_m\}$ be the collection of minimal prime ideals of the pure height $\ell $ ideal $(J_1,x_2,x_3, \ldots, x_\ell)$. If necessary, enlarge the set of height $\ell$ primes $\Lambda$ so that every component of $K$ is contained in a prime ideal of $\Lambda$. Let $W_\ell$ be the multiplicative set $R-\bigcup_{P\in \Lambda}P$. There exist elements $a,c\in K$ such that
\begin{enumerate}
\item $(a,c)R_{W_\ell}$ forms a reduction of $KR_{W_\ell}$;
\item the element $a$ generates $K$ at its minimal components;
\item as an ideal of $R$, the principal ideal $(a)=K\cap K'$ where $K'$ is of pure height $1$ whose components are disjoint from $K$, and the element $c$ avoids all components of $K'$.
\end{enumerate}
\end{claim}

\begin{proof}[Proof of claim] We are assuming the ideal $K$ has analytic spread at most $2$ at each of the localizations $R_{P_i}$ as $P_i$ varies among the prime ideals in $\Lambda=\{P_1,\ldots, P_m\}$. So for each $ 1 \le i\le m$ there exists $a_i,c_i\in K$ such that $(a_i,c_i)R_{P_i}$ forms a reduction of $KR_{P_i}$. For each $1 \le i\le m$ choose $r_i\in \bigcap_{P\in \Lambda-\{P_i\}}P-P_i$ and set $a'=\sum r_ia_i$ and $c'=\sum r_ic_i$. We claim $(a',c')R_{W_\ell}$ is a reduction of $KR_\ell$. By \cite[Proposition~8.1.1]{HunekeSwansonBook} it is enough to check $(a',c')$ forms a reduction of $K$ at each of the localizations $R_{P_i}$ for $1\leq i\leq m$. By \cite[Proposition~8.2.4]{HunekeSwansonBook} it is enough to check that the the fiber cone $R_P/PR_P\otimes R[Kt]\cong \bigoplus K^nR_{P_i}/P_i K^nR_{P_i}$ is finite over the subalgebra spanned by $((a',c')R_{P_i}, P_iK)/P_iK$. But $a'\equiv r_ia_i\mod P_iK$, $c'\equiv r_ic_i\mod P_iK$, $r_i$ is a unit of $R_{P_i}$, and therefore $(a',c')R_{W_\ell}$ does indeed form a reduction of $KR_{W_\ell}$ by a second application of \cite[Proposition~8.2.4]{HunekeSwansonBook}.

Now consider the set of primes $\Gamma=\{Q_1,\ldots, Q_n\}$ which are the minimal components of $K$. The purpose of enlarging the set of height $\ell$ primes in the statement of the claim was to insure that each $Q_j\in \Gamma$ is a prime ideal of the localization $R_{W_\ell}$. In particular, $(a',c')R_{Q_i}$ forms a reduction of $KR_{Q_i}$ for each $1\leq i\leq n$. But $R_{Q_i}$ is a discrete valuation ring and therefore for each $1\leq i\leq \ell$ either $KR_{Q_i}=(a')R_{Q_i}$ or $KR_{Q_i}=(c')R_{Q_i}$. Without loss of generality we assume that $KR_{Q_i}=(a')R_{Q_i}$ for at least one value of $i$ and relabel the primes in $\Gamma$ so that $KR_{Q_i}=(a')R_{Q_i}$ for each $1\leq i \leq j$ and $KR_{Q_i}\not=(a')R_{Q_i}$ for each $j+1\leq i\leq n$. Choose $r\in Q_1\cap \cdots \cap Q_j-\bigcup_{i=j+1}^n Q_i$ and consider the element $a'+r c'$. We claim that $a'+rc'$ generates $KR_{Q_i}$ for each $1\leq i\leq n$. First consider a localization at a prime $Q_i\in \Gamma$ with $1\leq i \leq j$. Then $(a',c')R_{Q_i}=(a')R_{Q_i}$ by assumption and so $(c')R_{Q_i}\subseteq (a')R_{Q_i}$. Because $r\in Q_i$ there is a strict containment of principal ideal $(rc')R_{Q_i}\subsetneq (a')R_{Q_i}$ and it follows that $(a')R_{Q_i}=(a'+rc')R_{Q_i}$. Now consider a localization $R_{Q_i}$ with $j+1\leq i\leq n$. We are assuming that $a'$ does not generate $KR_{Q_i}$ and therefore $(a')R_{Q_i}\subsetneq (c')R_{Q_i}=KR_{Q_i}$. Moreover, $r$ is a unit of $R_{Q_i}$ and therefore $(c')R_{Q_i}=(a'+rc')R_{Q_i}$.

Let $a=a'+rc'$. Then $(a,c')R_{W_\ell}=(a',c')R_{W_\ell}$ forms a reduction of $KR_{W_\ell}$ and the element $a$ generates $K$ at each of its minimal components as desired. Suppose as an ideal of $R$ the principal ideal $(a)$ has decomposition $(a)=K\cap K'\cap K''$ so that
\begin{enumerate}
\item $K,K',K''$ are pure height $1$ ideals whose components are disjoint from one another;
\item the components of $K'$ are height $1$ prime ideals which do not contain $c$;
\item the components of $K''$ are height $1$ prime ideals which do contain $c$.
\end{enumerate}
We take $K'$ or $K''$ to be $R$ if no such components of $(a)$ exist. If $K''=R$ then we let $c=c'$ and the elements $a,c$ satisfy the conclusions of the claim. If $K''\not=R$ then first observe that, because $(a,c')R_{W_\ell}$ forms a reduction of $KR_{W_\ell}$ and $a,c'\in K''$, we must have that $(a)R_{W_\ell}=(K\cap K')R_{W_\ell}$. Choose an element $r\in K\cap K'$ which avoids all components in $K''$ and consider the element $c=c'+r$. Then $(a,c)R_{W_\ell}=(a,c')R_{W_\ell}$ forms a reduction of $KR_{W_\ell}$. Moreover, the element $c$ avoids all minimal components of $K'$ and $K''$ by construction.
\end{proof}

 By assumption there exists a natural number $n_\ell$ so that $K^{n_\ell}R_{W_\ell}$ has reduction number at $1$ with respect to any reduction. Recall that $K^{n_\ell}$ and $K^{(n_\ell)}$ agree on the punctured spectrum. So we may replace $K$ by $K^{(n_\ell)}$, $x_2, x_3,\ldots,x_\ell$ by $x_2^{n_\ell},\ldots, x_\ell^{n_\ell}$, and $a,c$ by $a^{n_\ell},c^{n_\ell}$ and assume further that $(a,c)KR_{W_\ell}=K^2R_{W_\ell}$.

\begin{claim}There exists a parameter element $x_{\ell+1}$ of $R/(J_1,x_2,x_3,\ldots,x_\ell)$ such that the following hold:
\begin{enumerate}
\item $x_{\ell+1}^i$ annihilates $K^i/(a,c)^{i-1}K$ for every integer $i$;
\item $x_{\ell+1}$ annihilates $\Ext^{h+j+1}_S(R/(a,c)K,S)$ for every $\ell-1\leq j\leq d-2$;
\item $x_{\ell+1}$ annihilates $\Ext^{h+j+1}_S(R/K,S)$ for every $\ell-1\leq j\leq d-2$.
\end{enumerate}
\end{claim}

\begin{proof} Consider $W_\ell$ as a multiplicative set of $S.$ Then $S_{W_\ell}$ has dimension $h+\ell$, $K^iR_{W_\ell}=K^{(i)}S_{W_\ell}$, and $(a,c)KS_{W_{\ell}}=K^{(2)}S_{W_\ell}$. Because $K^{(i)}$ is an unmixed ideal we have that $R_{W_\ell}/K^{(i)}R_{W_\ell}$ has positive depth and therefore the $\Ext$-modules 
\[
\Ext^{h+j+1}_S(R/K)\otimes R_{W_{\ell}}\mbox{ and } \Ext^{h+j+1}_S(R/(a,c)K)\otimes R_{W_{\ell}}
\] 
are $0$ for each $\ell-1\leq j\leq d-2$. It follows that we can choose $x_{\ell+1}$ a parameter on $R/(J_1,x_2,x_3,\ldots,x_\ell)$ such that  $x_{\ell+1}K^2\subseteq (a,c)K$ and $x_{\ell+1}$ satisfies $(2)$ and $(3)$. Because $x_{\ell+1}K^2\subseteq (a,c)K$ it follows that for every $i\geq 1$ that $x_{\ell+1}^{i-1}K^{i}\subseteq (a,c)^{i-1}K$ and therefore $(1)$ is satisfied as well.
\end{proof}

The element $x_{\ell+1}^{i-1}$ annihilates the left term of the following short exact sequence:
\[
0\to \frac{K^i}{(a,c)^{i-1}K}\to \frac{R}{(a,c)^{i-1}K}\to \frac{R}{K^i}\to 0.
\]
In particular, there are exact sequences
\[
\Ext^{h+j}_S(K^i/(a,c)^{i-1}K,S)\to\Ext^{h+j+1}_S(R/K^i,S)\to \Ext^{h+j+1}_S(R/(a,c)^{i-1}K,S)
\]
and the left term is annihilated by $x_{\ell+1}^{i-1}$. We will show that $x_{\ell+1} \Ext^{h+j+1}_S(R/(a,c)^{i-1}K,S)=0$ for every $i\geq 2$ and $\ell-1\leq j\leq d-2$. It will then follow that $x_{\ell+1}^i$ annihilates $\Ext^{h+j+1}_S(R/K^i,S)$ for every $i$ and $\ell-1\leq j\leq d-2$ as desired.

\begin{claim}\label{SES claim} For every integer $i$ there is short exact sequence
\[
0\to \frac{R}{(ac^{i})}\to \frac{R}{a(a,c)^{i-1}K}\oplus \frac{R}{c^iK}\to \frac{R}{(a,c)^iK}\to 0.
\]
\end{claim}

\begin{proof}[Proof of claim]
For any two ideals $I,J\subseteq R$ there is a short exact sequence
\[
0\to \frac{R}{I\cap J}\to \frac{R}{I}\oplus \frac{R}{J}\to \frac{R}{I+J}\to 0.
\]
Therefore it is enough to show that 
\[
a(a,c)^{i-1}K\cap c^iK=(ac^i).
\]
Clearly $ac^{i}\in a(a,c)^{i-1}K\cap c^iK$. Now consider an element of the form $c^{i}y$ with $y\in K$ and $c^{i}y\in a(a,c)^{i-1}K$. To show $c^{i}y\in (ac^{i})$ we only need to show $y\in (a)$. Recall that by Claim~\ref{analytic spread claim} we have that $(a)=K\cap K'$ and $c$ avoids all components of $K'$.  We already know that $y\in K$. Localizing at a component $P$ of $K'$ we have that 
\[
c^{i}y\in a(a,c)^{i-1}KR_P.
\]
However, $c$ is a unit of $R_P$, $c\in K$, and therefore $y\in a R_P$.
\end{proof}

\begin{claim}\label{Claim isomorphisms} For each $2\leq j\leq d-2$ there are isomorphisms 
\[
\Ext^{h+j+1}_S(R/a(a,c)^{i-1}K,S)\cong \Ext^{h+j+1}_S(R/(a,c)^{i-1}K,S)
\]
and
\[
\Ext^{h+j+1}_S(R/c^{i}K,S)\cong\Ext^{h+j+1}_S(R/K,S).
\]
\end{claim}

\begin{proof}[Proof of claim]
For the first isomorphism consider the long exact sequence of $\Ext$-modules induced from the short exact sequence
\[
0\to \frac{R}{(a,c)^{i-1}K}\xrightarrow{\cdot a}\frac{R}{a(a,c)^{i-1}K}\to \frac{R}{(a)}\to 0
\]
and for the second isomorphism consider the long exact sequence of $\Ext$-modules induced from the short exact sequence
\[
0\to \frac{R}{K}\xrightarrow{\cdot c^{i}} \frac{R}{c^{i}K}\to \frac{R}{(c^{i})}\to 0.
\]
\end{proof}
Observe that by Claim~\ref{SES claim} there are isomorphisms
\[
\Ext^{h+j+1}_S(R/(a,c)^iK,S)\cong \Ext^{h+j+1}_S(R/a(a,c)^{i-1}K,S)\oplus \Ext^{h+j+1}_S(R/c^iK,S)
\]
for all $2\leq j\leq d-2$. Therefore Claim~\ref{Claim isomorphisms} and induction we find that there isomorphisms
\[
\Ext^{h+j+1}_S(R/(a,c)^iK,S)\cong \bigoplus \Ext^{h+j+1}_S(R/K,S)
\]
The element $x_{\ell+1}$ has the property that it annihilates the modules appearing the direct sum decompositions above. Therefore $x_{\ell+1}$ annihilates each $\Ext^{h+j+1}_S(R/(a,c)^iK,S)$ for each $\ell-1\leq j \leq d-2$ as desired.
\end{proof}

Theorem~\ref{Main tight closure theorem} is established by piecing together Theorem~\ref{theorem how to make test ideals agree using lcbs}, Proposition~\ref{proposition on how to make test ideals agree}, and Theorem~\ref{theorem when are technical condtions met}.

\begin{corollary}\label{Main prime characteritic corollary} Let $(R,\fm,k)$ be an excellent local normal Cohen-Macaulay domain of prime characteristic $p>0$, of Krull dimension at least $4$, and $\Q$-Gorenstein in codimension $2$. Suppose that some symbolic power of the anti-canonical ideal of $R$ has analytic spread no more than $2$ on the punctured spectrum. Then $0^*_{E_R(k)}=0^{*,fg}_{E_R(k)}$.
\end{corollary}

\begin{proof} We wish to invoke Theorem~\ref{theorem when are technical condtions met}. Therefore if we denote by $K\subseteq R$ an anti-canonical ideal we must prove the existence of an integer $n$ such that for all $P\in \Spec(R)\setminus\{\fm\}$ that $K^{(n)}R_P$ has a reduction by $2$ elements with reduction number $1$.

 The anti-canonical algebra $\mathcal{R}_K$ is Noetherian on the punctured spectrum by Theorem~\ref{Criterion for finite generation}. Hence $R_P$ is strongly $F$-regular by \cite[Corollary~5.9]{CEMS} for each $P\in \Spec(R)\setminus\{\fm\}$. Therefore at each prime ideal $P\in \Spec(R)\setminus\{\fm\}$ there exists an integer $n_P$ so that $K^{(n_P)}R_P$ has analytic spread $2$ and reduction number $1$ with respect to any reduction by Theorem~\ref{Linquan's Theorem}. For a choice of reduction of $K^{(n_P)}R_P$ it is easy to see there is an open neighborhood of $P\in \Spec(R)\setminus\{\fm\}$ so that $K^{(n_P)}$ has a reduction by $2$ elements with reduction number $1$. By a simple quasi-compactness argument there exists an integer $n$ such that $K^{(n)}R_P$ has a reduction by $2$ elements with reduction number $1$ for each $P\in \Spec(R)\setminus\{\fm\}$ and therefore Theorem~\ref{theorem when are technical condtions met} is applicable.

Let $J_1\subsetneq R$ be a choice of a canonical ideal and let $x_1\in J_1$ be a generic generator. By Theorem~\ref{theorem when are technical condtions met} we may extend $x_1$ to a suitable system of parameters $x_1,x_2,\ldots, x_d$ such that 
\[
(x_2^i,x_3^i,\ldots, x_{j+2}^i)\Ext^{h+j}_S(\Ext^{h+1}_S(R/J_1^{mi+1},S),S)=0
\]
for every integer $i\in \N$ and $2\leq j\leq d-2$. By Proposition~\ref{proposition on how to make test ideals agree} for every integer $i\in \N$ there exists an integer $\ell$ such that
\[
\lcb_{d-1}(x_2^{\ell(d-1)},x_3^{\ell(d-1)},x^{d-1}_4,\ldots,x^{d-1}_d;R/J_1^{(mi+1)})\leq i+1.
\]
We replace $x_2$ and $x_3$ by powers of themselves and may assume there exists elements $a_2\in J$ and $a_3\in J^{(m)}$ such that $x_2J\subseteq a_2R$ and $x_3J^{(m)}\subseteq a_3R$. 

We can further replace $x_2,\ldots, x_d$ by the sequence of elements $x_2^{d-1},\ldots, x_d^{d-1}$ and have now have that for all $i\in \N$ there exists an integer $\ell$ such that
\[
\lcb_{d-1}(x_2^{\ell},x_3^{\ell},x_4,\ldots,x_d;R/J_1^{(mi+1)})\leq i+1.
\]
The corollary now follows by Theorem~\ref{theorem how to make test ideals agree using lcbs}.
\end{proof}

\begin{corollary}\label{Corollary weak implies strong} Let $R$ be a locally excellent weakly $F$-regular ring of prime characteristic $p$ which has a canonical ideal. Suppose further that at each non-closed point of $\Spec(R)$  there is a symbolic power of the anti-canonical ideal which has analytic spread  at most $2$. Then $R$ is strongly $F$-regular.
\end{corollary}

\begin{proof} It is well known that the properties of being weakly $F$-regular and strongly $F$-regular can be checked at localizations at the maximal ideals of $R$, see \cite[Corollary~4.15]{HHJAMS}. The properties of weakly $F$-regular and strongly $F$-regular for a local ring can be checked after completion. In which case, the property of being weakly $F$-regular is equivalent to $0^{*,fg}_{E_R(k)}$ being $0$ and the property of being strongly $F$-regular is equivalent is $0^*_{E_R(k)}$ being $0$. Every complete local weakly $F$-regular ring is normal by \cite[Lemma~5.9]{HHJAMS}, Cohen-Macualay by \cite[Therorem~4.9]{HHJAMS}, and $\Q$-Gorenstein in codimension $2$ by \cite[Theorem~3.1]{SmithRationalSingularities} and \cite[Proposition~17.1]{LipmanRationalSingularities}. In particular, Corollary~\ref{Main prime characteritic corollary} is applicable and the result follows.
\end{proof}

\begin{corollary}\label{$F$-regular implies strong dimension 4} Let $R$ be a Cohen-Macaulay local domain of Krull dimension no more than $4$ which is essentially of finite type over a field $
\mathscr{K}$ of prime characteristic $p>5$. Suppose that $R$ is $F$-regular. Then $R$ is strongly $F$-regular.
\end{corollary}

\begin{proof} Without loss of generality we may assume $R=(R,\fm,k)$ is local. We are assuming $R$ is weakly $F$-regular at each $P\in \Spec(R)-\{\fm\}$. If $P\in \Spec(R)$ is height $3$ then $R_P$ is strongly $F$-regular by \cite{Williams}. Using gamma constructions with respect to the complete local ring $\mathscr{K}$, we may assume $R$ is $F$-finite, see \cite[Section~6 and Theorem~7.24]{HHTAMS} and \cite[Corollary~3.31]{Hashimoto}. By \cite[Corollary~6.9]{SchwedeSmith} there exists an effective boundary divisor $\Delta$ such that $(\Spec(R_P),\Delta)$ is globally $F$-regular (or just $F$-regular since $\Spec(R_P)$ is affine) and therefore has KLT singularities by \cite[Theorem~3.3]{HaraWatanabe}. Utilizing \cite[Corollary~1.12]{DasWaldronArxiv} we have that for each non-closed point of $R$, the anti-canonical algebra of $R_P$ is Noetherian. This is then equivalent to some symbolic power of the anti-canonical algebra being principal in codimension $2$ and having analytic spread at most $2$ at each non-closed point of $\Spec(R)$ by Theorem~\ref{Criterion for finite generation}. Therefore $0^{*,fg}_{E_R(k)}=0^*_{E_R(k)}$ by Corollary~\ref{Main prime characteritic corollary}.
\end{proof}

Recall that Murthy proved the notions of weakly $F$-regular and $F$-regular agree for rings finite type over an uncountable field. Therefore we can equate the notions of weakly $F$-regular and strongly $F$-regular for a new, large, and interesting class of four dimensional rings.

\begin{corollary}\label{Extending Murthy's theorem} Let $R$ be finite type over a field $\mathscr{K}$ of prime characteristic $p>5$, of Krull dimension no more than $4$, and assume that $\mathscr{K}$ has infinite transcendence degree over $\mathbb{F}_p$. If $R$ is weakly $F$-regular then $R$ is strongly $F$-regular.
\end{corollary}

\begin{proof}
The notions of weakly $F$-regular and $F$-regular are equivalent for rings of finite type over fields which have infinite transcendence degree over the prime field, \cite[Theorem~8.1]{HHTAMS}. Therefore the notions of weakly $F$-regular and strongly $F$-regular are equivalent for such rings by Corollary~\ref{$F$-regular implies strong dimension 4}.
\end{proof}

We end this section with some remarks concerning the assumptions of Theorem~\ref{theorem when are technical condtions met}. As of now, we can only equate $F$-regular and strongly $F$-regular rings which are $4$ dimensional and essentially of finite type over a ring of prime characteristic $p>5$, we can not equate weakly $F$-regular with strongly $F$-regular for such rings. Unlike the $3$ dimensional case, it is not clear at all if the property of being weakly $F$-regular localizes. We do however know the property of being $F$-rational localizes, which is all that is needed in the three dimensional case to invoke Lipman's results from \cite{LipmanRationalSingularities} to know symbolic Rees rings of pure height $1$ ideals are Noetherian on the punctured spectrum. In dimension $4$ we only know that $R$ is $F$-rational and $F$-split at $3$ dimensional points and it is unlikely that symbolic Rees rings of pure height $1$ will be Noetherian for such rings. See \cite{CutkoskyDuke} for an example of a $3$-dimensional rational singularity defined over $\mathbb{C}$ which admits pure height $1$ ideals whose symbolic Rees ring is non-Noetherian.

\section{$F$-signature and relative Hilbert-Kunz multiplicity}\label{section $F$-signature and relative Hilbert-Kunz multiplicity}

\subsection{Background on $F$-signature and Hilbert-Kunz multiplicity} We summarize some basic properties of Frobenius splitting numbers, $F$-signature, and Hilbert-Kunz multiplicity. For an introduction to these concepts we refer the reader to \cite{HunekeHK, PolstraTucker}. Let $(R,\fm,k)$ be a local $F$-finite domain of prime characteristic $p>0$ and Krull dimension $d$. For each $e\in \N$ let $a_e(R)$ be the largest rank of a free $R$-module $G$ for which there exists an onto $R$-linear map $F^e_*R\to G$. The $F$-signature of $R$ is the limit 
\[
s(R)=\lim_{e\to \infty}\frac{a_e(R)}{\rank_R(F^e_*R)},
\]
a limit which always exists by \cite[Main Result]{Tucker}. The ring $R$ is strongly $F$-regular if and only if $s(R)>0$ by \cite[Main Theorem]{AberbachLeuschke}. For each integer $e\in \N$ we denote by $I_e$ the $e$th Frobenius degeneracy ideal. Specifically,
\[
I_e=\{r\in R\mid \varphi(F^e_*r)\in \fm, \forall\varphi\in\Hom_R(F^e_*R,R)\}.
\]
The ideals $I_e$ satisfy the following properties:
\begin{enumerate}
\item $\fm^{[p^e]}\subseteq I_e$;
\item For each integer $e_0\in \N$, $I_e^{[p^{e_0}]}\subseteq I_{e+e_0}$;
\item $\displaystyle \frac{a_e(R)}{\rank(F^e_*R)}=\frac{\lambda(R/I_e)}{p^{ed}}$;
\item $\displaystyle s(R)=\lim_{e\to \infty}\frac{\lambda(R/I_e)}{p^{e\dim(R)}}$.
\end{enumerate}

Suppose $I\subseteq R$ is an $\fm$-primary ideal. The Hilbert-Kunz multiplicity of the ideal $I\subseteq R$ is the limit 
\[
\e(I)=\lim_{e\to \infty}\frac{\lambda(R/I^{[p^e]})}{p^{ed}},
\]
a limit which exists by \cite[Theorem~1.8]{Monsky}. By \cite[Proposition~2.1]{Kunz2} we have that for each $\fm$-primary ideal $I\subseteq R$,
\[
\frac{\lambda(R/I^{[p^e]})}{p^{ed}}=\frac{\lambda(F^e_*R/IF^e_*R)}{\rank(F^e_*R)}.
\]
Therefore the Hilbert-Kunz multiplicity of an $\fm$-primary ideal agrees with the limit
\[
\e(R)=\lim_{e\to \infty }\frac{\lambda(F^e_*R/IF^e_*R)}{\rank(F^e_*R)}.
\]
Suppose that $F^e_*R\cong R^{\oplus a_e(R)}\oplus M_e$. Then for each $\fm$-primary ideal $\lambda(F^e_*R/IF^e_*R)=a_e(R)\lambda(R/I)+\lambda(M_e/IM_e)$. If $I\subsetneq J$ are $\fm$-primary it is then easy to see that
\[
a_e(R)\lambda(J/I)\leq \lambda(F^e_*R/IF^e_*R)-\lambda(F^e_*R/JF^e_*R)
\]
and therefore for each pair of $\fm$-primary ideals $I\subsetneq J$ we have that
\[
s(R)\leq \frac{\e(I)-\e(J)}{\lambda(J/I)}.
\]
Work of the second author and Tucker show that $F$-signature of a local ring is realized as the infimum of relative Hilbert-Kunz multiplicities.

\begin{theorem}{\cite[Theorem~A]{PolstraTucker}}
If $(R,\fm,k)$ is an $F$-finite local ring, then
 \[
 s(R) = \inf_{\substack{I \subseteq J  \subseteq R, \;\lambda(R/I) < \infty  \\ I \neq J, \; \lambda(R/J) < \infty} }\frac{e_{HK}(I) - e_{HK}(J)}{\lambda(J/I)} = \inf_{\substack{I \subseteq R, \; \lambda(R/I) < \infty \\  x \in R, \; ( I : x )=\fm}} e_{HK}(I) - e_{HK}((I,x)).
 \]
\end{theorem}

\subsection{$F$-signature and relative Hilbert-Kunz multiplicity}

Our proof of Theorem~\ref{Main theorem prime char invariants} begins with the following well known lemma concerning the Frobenius splitting numbers of a local ring. We refer the reader to \cite[Lemma~6.2]{PolstraTucker} for a direct proof.

\begin{lemma}\label{Lemma on Frobenius splitting numbers} Let $(R,\fm,k)$ be an $F$-finite local domain of prime characteristic $p>0$ and Krull dimension $d$. Suppose that $J_1\subsetneq R$ is a choice of canonical ideal, $0\neq x_1\in J_1$, $x_2,\ldots,x_d$ parameters on $R/J_1$, and $u\in R$ generates the socle mod $(J_1,x_2,\ldots, x_d)$. For each integer $t\in \N$ let $I_t=(x_1^{t-1}J_1,x_2^t,\ldots,x_d^t)$ and $u_t=u(x_1\cdots x_d)^{t-1}$. Then for each $e\in \N$ the sequence of ideals $\{(I_t^{[p^e]}:u_t^{p^e})\}_{t\in \N}$ forms an ascending chain of ideals which stabilizes at the Frobenius degeneracy ideal $I_e$. In particular, for each $e\in \N$ there exists a $t\in \N$ such that
\[
\frac{a_e(R)}{\rank(F^e_*R)}=\frac{\lambda(R/(I_t,u_t)^{[p^e]})-\lambda(R/(I_t,u_t)^{[p^e]})}{p^{ed}}.
\]
\end{lemma}


\begin{theorem}\label{Theorem Watanabe-Yoshida problem}
Let $(R,\fm,k)$ be a local strongly $F$-regular $F$-finite domain of prime characteristic $p>0$ such that some symbolic power of the anti-canonical ideal has analytic spread at most $2$ on the punctured spectrum. Then there exists an irreducible $\fm$-primary ideal $I$ and $u\in R$ which generates the socle mod $I$ such that for each integer $e\in \N$
\[
I_e=(I^{[p^e]}:u^{p^e}).
\]
It follows that for all $e\in \N$
\[
\frac{a_e(R)}{\rank(F^e_*R)}=\frac{\lambda(R/I^{p^e})-\lambda(R/(I,u)^{[p^e]})}{p^{e{\dim R}}}
\]
and therefore
\[
s(R)=\e(I)-\e((I,u)).
\]
\end{theorem}

\begin{proof} 
 Following the proof of Theorem~\ref{theorem how to make test ideals agree using lcbs} and utilizing Theorem~\ref{theorem when are technical condtions met}, Proposition~\ref{proposition on how to make test ideals agree}, and Proposition~\ref{proposition technical lcb}, if $J_1\subsetneq R$ is a choice of canonical ideal there exists $0\not=x_1\in J_1$, parameters $x_2,\ldots,x_d$ on $R/J_1$ and $m\in \N$ such that if we let $\{I_t\}, \{u_t\}$ be as in Lemma~\ref{Lemma on Frobenius splitting numbers} then for each integer $t\in \N$
\begin{align*}
(I_t^{[p^e]}:u_t^{p^e})=(J_1^{[p^e]},x_2^{tp^e},\cdots, x_d^{tp^e}):u(x_2\cdots x_d)^{(t-1)p^e}\subseteq \\
(J_1^{[p^e]},x^{4 p^e}_2,x_3^{4 p^e},x_4^{4p^e},\ldots, x_d^{4p^e}):x_1^{4m-1}u(x^3_2x^3_3(x_4\cdots x_d)^3)^{p^e}=(I_3^{[p^e]}:u_3^{p^e}):x_1^{4m-1}.
\end{align*}

Because the sequence of ideals $\{(I_t^{[p^e]}:u_t^{p^e})\}$ is an ascending chain of ideals which stabilizes at the Frobenius degeneracy ideal $I_e$ we see that there are containments
\[
(I_3^{[p^e]}:u_3^{p^e})\subseteq I_e\subseteq (I_3^{[p^e]}:u_3^{p^e}):x_1^{4m-1}
\]
We claim that the inclusion $(I_3^{[p^e]}:u_3^{p^e})\subseteq I_e$ is an equality for each $e\in\N$. Suppose $r\in I_e$, then $r^{p^{e_0}}\in I_e^{[p^{e_0}]}\subseteq I_{e+e_0}\subseteq (I_3^{[p^{e+e_0}]}:u_3^{p^{e+e_0}}):x_1^{4m-1}$. Therefore $x_1^{4m+1}(u_3^{p^e}r)^{p^{e_0}}\in (I_3^{[p^e]})^{[p^{e_0}]}$ for all $e_0\in \N$. Hence $u_3^{p^e}r\in (I_3^{[p^{e}]})^*=I_3^{[p^e]}$, i.e. $r\in (I_3^{[p^e]}:u_3^{p^e})$ as claimed.
\end{proof}

Suppose $(R,\fm,k)$ is an $F$-finite normal domain of Krull dimension $d$. Then for each $\fm$-primary ideal $I\subseteq R$ there is a real number $\beta_I$ such that $\lambda(R/I^{[p^e]})=\e(R)p^{ed}+\beta_Ip^{e(d-1)}+O(p^{e(d-2)})$ by \cite[Theorem~1]{HunekeMcDermottMonsky}.

\begin{corollary} Let $(R,\fm,k)$ be a local strongly $F$-regular $F$-finite domain of prime characteristic $p>0$, of Krull dimension $d$, and such that some symbolic power of the anti-canonical ideal has analytic spread at most $2$ on the punctured spectrum. Then there exists a real number $\tau\in \mathbb{R}$ such that
\[
\lambda(R/I_e)=s(R)p^{ed}+\tau p^{e(d-1)}+O(p^{e(d-2)}).
\]
\end{corollary}

\begin{proof}
By Theorem~\ref{Theorem Watanabe-Yoshida problem} there exists $\fm$-primary ideal $I\subseteq R$ and $u\in R$ such that 
\[
\frac{a_e(R)}{\rank(F^e_*R)}=\frac{\lambda(R/I_e)}{p^{ed}}=\frac{\lambda(R/I^{[p^e]})-\lambda(R/(I,u)^{[p^e]})}{p^{ed}}
\]
for all $e\in \N$. Every strongly $F$-regular local ring is normal and therefore the results of \cite{HunekeMcDermottMonsky} are applicable.
\end{proof}

\section{Questions}\label{section open problems}

Lyubeznik and Smith proved that if $R$ is an $F$-finite $\mathbb{N}$-graded ring then the finitistic test ideal and test ideal of $R$ agree, \cite[Corollary~3.4]{LyubeznikSmith}. This article shows equality of test ideals for local rings whenever a certain family of $\Ext$-modules are annihilated in a controlled way. It is therefore natural to ask when the $\Ext$-annihilation properties established in Theorem~\ref{theorem when are technical condtions met} hold for graded rings. For example, we ask the following:

\begin{question} Let $S=k[T_1,\ldots,T_n]$ be a polynomial ring over a field $k$ of prime characteristic $p>0$, $P\subseteq R$ a homogeneous prime ideal of height $h$, and $S=R/P$. Suppose that the Krull dimension of $R$ is $d$ and $J_1\subsetneq R$ a canonical ideal.  Does there exist an integer $m\in \N$ and parameters $x_2,\ldots ,x_d$ on $R/J_1$ such that
\[
(x^i_2,x^i_3,\ldots ,x^i_{j+2})\Ext^{h+j}_S(\Ext^{h+1}_S(R/J_1^{mi+1},S),S)=0
\]
for every integer $i \in \N$ and $2\leq j\leq d-2$?
\end{question}

Under mild hypotheses, this article equates the finitistic test ideal and test ideal of a ring under the assumption that the anti-canonical ideal has analytic spread at most $2$ on the punctured spectrum. For rings of Krull dimension at most $4$ this is equivalent to the anti-canonical algebra being Noetherian on the punctured spectrum.

\begin{question} Can the techniques of this article be extended to show equality of test ideals whenever the anti-canonical algebra of a ring is assumed to be Noetherian on the punctured spectrum?
\end{question}

The critical point of the argument where the analytic spread $2$ assumption is being used is in Claim~\ref{SES claim}. In Claim~\ref{SES claim} we find families of ideals which intersect principally, so that the higher $\Ext$-modules of the cyclic modules defined by these ideals vanish.

There are interesting connections between the theory of multiplier ideals in the birational geometry of complex varieties and test ideals of varieties defined over a field of prime characteristic. Suppose $R$ is an $F$-finite normal domain. Following the methods of \cite{HaraYoshida, Takagi} one can develop a tight closure theory of triples $(R,\Delta,\fa^t)$ where $\Delta\geq 0$ is an effective $\Q$-divisor, $\fa\subseteq R$ an ideal, and $t\geq 0$ a real number. Suppose that $K$ is the fraction field of $R$. Then for each $e\in \N$ consider the fractional ideal $R((p^e-1)\Delta)\subseteq K$ generated by nonzero elements $f\in K$ such that $\mbox{div}(f)+(p^e-1)\Delta$ is effective. For each $e\in \N$ we consider the extension of scalars functors $F^e_\Delta: \Mod(R)\to \Mod(R)$ sending a module $M\mapsto {^{e}R((p^e-1)\Delta)}\otimes_R M$. An element $m\in M$ is mapped to $F^e_\Delta(m)=m^{p^e}:=1\otimes m\in {^{e}R((p^e-1)\Delta)}\otimes_R M$. If $N\subseteq M$ are $R$-modules we say that an element $m$ is in the $(\Delta,\fa^t)$-tight closure of $N$, denoted by $N_M^{(\Delta,\fa^t)*}$, if there exists $c\in R^\circ$ such that the submodule $\fa^{\lceil tp^e\rceil}m$ of $M$ is contained in the kernel of the following maps for all $e\gg0$;
\[
M\to M/N\to F^{e}_{\Delta}(M/N)\xrightarrow{\cdot c}F^{e}_{\Delta}(M/N).
\]
The finitistic $(\Delta,\fa^t)$-tight closure of $N\subseteq M$ is denoted by $N^{(\Delta,\fa^t)*,fg}_M$ and is $\bigcup (N\cap M')_{M'}^{(\Delta,\fa^t)*}$ where the union is taken over all finitely generated submodules $M'$ of $M$. If $\Delta=0$ and $\fa=R$ then $(\Delta,\fa^t)$-tight closure agrees with the usual tight closure and finitistic $(\Delta,\fa^t)$-tight closure agrees with the usual finitistic tight closure.

\begin{question} To what extent can the results of this article be extended to show equality of test ideals of pairs or triples? Specifically, if $(R,\Delta,\fa^t)$ is a triple and $(R,\fm,k)$ is local, then when may we conclude that 
\[
0^{(\Delta,\fa^t)*}_{E_R(k)}=0^{(\Delta,\fa^t)*,fg}_{E_R(k)}?
\] 
\end{question}

For a partial answer to the above question see \cite[Theorem~2.8]{Takagi} for a proof that $0^{(\Delta,\fa^t)*}_{E_R(k)}=0^{(\Delta,\fa^t)*,fg}_{E_R(k)}$ when $\fa=R$ and $K_X+\Delta$ is assumed to be a $\Q$-Cartier divisor, where $K_X$ is a canonical divisor on $X=\Spec(R)$.

\begin{center}{Acknowledgements} \end{center}

We thank Chris Hacon, Karl Schwede, Anurag Singh, and Joe Waldron for valuable correspondence during the duration of this project. We thank Jonathan Monta\~{n}o for useful discussions concerning analytic spread and symbolic powers of ideals. We especially thank Linquan Ma for feedback on previous drafts of this article. We are remarkably grateful for Craig Huneke and Sarasij Maitra for detailed discussions concerning the material of the article which led to an improved exposition of the main arguments.

\bibliographystyle{alpha}
\bibliography{References}

\newcommand{\etalchar}[1]{$^{#1}$}
\begin{thebibliography}{MPST19}

\bibitem[Abe02]{Aberbach2002}
Ian~M. Aberbach.
\newblock Some conditions for the equivalence of weak and strong
  {$F$}-regularity.
\newblock {\em Comm. Algebra}, 30(4):1635--1651, 2002.

\bibitem[AHH93]{AHH}
Ian~M. Aberbach, Melvin Hochster, and Craig Huneke.
\newblock Localization of tight closure and modules of finite phantom
  projective dimension.
\newblock {\em J. Reine Angew. Math.}, 434:67--114, 1993.

\bibitem[AL03]{AberbachLeuschke}
Ian~M. Aberbach and Graham~J. Leuschke.
\newblock The {$F$}-signature and strong {$F$}-regularity.
\newblock {\em Math. Res. Lett.}, 10(1):51--56, 2003.

\bibitem[BH93]{BrunsHerzog}
Winfried Bruns and J\"{u}rgen Herzog.
\newblock {\em Cohen-{M}acaulay rings}, volume~39 of {\em Cambridge Studies in
  Advanced Mathematics}.
\newblock Cambridge University Press, Cambridge, 1993.

\bibitem[BM10]{BrennerMonsky}
Holger Brenner and Paul Monsky.
\newblock Tight closure does not commute with localization.
\newblock {\em Ann. of Math. (2)}, 171(1):571--588, 2010.

\bibitem[Bro79]{Brodmann}
M.~Brodmann.
\newblock Asymptotic stability of {${\rm Ass}(M/I\sp{n}M)$}.
\newblock {\em Proc. Amer. Math. Soc.}, 74(1):16--18, 1979.

\bibitem[CEMS18]{CEMS}
Alberto Chiecchio, Florian Enescu, Lance~Edward Miller, and Karl Schwede.
\newblock Test ideals in rings with finitely generated anti-canonical
  algebras---corrigendum [ {MR}3742559].
\newblock {\em J. Inst. Math. Jussieu}, 17(4):979--980, 2018.

\bibitem[CHS10]{CutkoskyHerzogSrinivasan}
Steven~Dale Cutkosky, J\"{u}rgen Herzog, and Hema Srinivasan.
\newblock Asymptotic growth of algebras associated to powers of ideals.
\newblock {\em Math. Proc. Cambridge Philos. Soc.}, 148(1):55--72, 2010.

\bibitem[Cut88]{CutkoskyDuke}
S.~Cutkosky.
\newblock Weil divisors and symbolic algebras.
\newblock {\em Duke Math. J.}, 57(1):175--183, 1988.

\bibitem[DM19]{DaoMontano}
Hailong {Dao} and Jonathan {Monta{\~n}o}.
\newblock {Symbolic analytic spread: upper bounds and applications}.
\newblock {\em arXiv e-prints}, page arXiv:1907.07081, Jul 2019.

\bibitem[Dut13]{DuttaI}
S.~P. Dutta.
\newblock The monomial conjecture and order ideals.
\newblock {\em J. Algebra}, 383:232--241, 2013.

\bibitem[Dut16]{DuttaII}
S.~P. Dutta.
\newblock The monomial conjecture and order ideals {II}.
\newblock {\em J. Algebra}, 454:123--138, 2016.

\bibitem[DW19]{DasWaldronArxiv}
Omprokash {Das} and Joe {Waldron}.
\newblock {On the log minimal model program for $3$-folds over imperfect fields
  of characteristic $p>5$}.
\newblock {\em arXiv e-prints}, page arXiv:1911.04394, Nov 2019.

\bibitem[Har77]{Hartshorne}
Robin Hartshorne.
\newblock {\em Algebraic geometry}.
\newblock Springer-Verlag, New York-Heidelberg, 1977.
\newblock Graduate Texts in Mathematics, No. 52.

\bibitem[Has10]{Hashimoto}
Mitsuyasu Hashimoto.
\newblock {$F$}-pure homomorphisms, strong {$F$}-regularity, and
  {$F$}-injectivity.
\newblock {\em Comm. Algebra}, 38(12):4569--4596, 2010.

\bibitem[HH89]{HHStrongFregular}
Melvin Hochster and Craig Huneke.
\newblock Tight closure and strong {$F$}-regularity.
\newblock {\em M\'{e}m. Soc. Math. France (N.S.)}, (38):119--133, 1989.
\newblock Colloque en l'honneur de Pierre Samuel (Orsay, 1987).

\bibitem[HH90]{HHJAMS}
Melvin Hochster and Craig Huneke.
\newblock Tight closure, invariant theory, and the {B}rian\c{c}on-{S}koda
  theorem.
\newblock {\em J. Amer. Math. Soc.}, 3(1):31--116, 1990.

\bibitem[HH91]{HHsmall}
Melvin Hochster and Craig Huneke.
\newblock Tight closure and elements of small order in integral extensions.
\newblock {\em J. Pure Appl. Algebra}, 71(2-3):233--247, 1991.

\bibitem[HH93]{HHPhantom}
Melvin Hochster and Craig Huneke.
\newblock Phantom homology.
\newblock {\em Mem. Amer. Math. Soc.}, 103(490):vi+91, 1993.

\bibitem[HH94a]{HHTAMS}
Melvin Hochster and Craig Huneke.
\newblock {$F$}-regularity, test elements, and smooth base change.
\newblock {\em Trans. Amer. Math. Soc.}, 346(1):1--62, 1994.

\bibitem[HH94b]{HHJAG}
Melvin Hochster and Craig Huneke.
\newblock Tight closure of parameter ideals and splitting in module-finite
  extensions.
\newblock {\em J. Algebraic Geom.}, 3(4):599--670, 1994.

\bibitem[HHT07]{HerzogHibiTrung}
J\"{u}rgen Herzog, Takayuki Hibi, and Ng\^{o}~Vi\^{e}t Trung.
\newblock Symbolic powers of monomial ideals and vertex cover algebras.
\newblock {\em Adv. Math.}, 210(1):304--322, 2007.

\bibitem[HL02]{HunekeLeuschke}
Craig Huneke and Graham~J. Leuschke.
\newblock Two theorems about maximal {C}ohen-{M}acaulay modules.
\newblock {\em Math. Ann.}, 324(2):391--404, 2002.

\bibitem[HMM04]{HunekeMcDermottMonsky}
Craig Huneke, Moira~A. McDermott, and Paul Monsky.
\newblock Hilbert-{K}unz functions for normal rings.
\newblock {\em Math. Res. Lett.}, 11(4):539--546, 2004.

\bibitem[Hoa93]{Hoa}
Le~Tuan Hoa.
\newblock Reduction numbers and {R}ees algebras of powers of an ideal.
\newblock {\em Proc. Amer. Math. Soc.}, 119(2):415--422, 1993.

\bibitem[Hoc77]{HochsterPurity}
Melvin Hochster.
\newblock Cyclic purity versus purity in excellent {N}oetherian rings.
\newblock {\em Trans. Amer. Math. Soc.}, 231(2):463--488, 1977.

\bibitem[HS06]{HunekeSwansonBook}
Craig Huneke and Irena Swanson.
\newblock {\em Integral closure of ideals, rings, and modules}, volume 336 of
  {\em London Mathematical Society Lecture Note Series}.
\newblock Cambridge University Press, Cambridge, 2006.

\bibitem[HS15]{HunekeSmrinov}
Craig Huneke and Ilya Smirnov.
\newblock Prime filtrations of the powers of an ideal.
\newblock {\em Bull. Lond. Math. Soc.}, 47(4):585--592, 2015.

\bibitem[Hun96]{HunekeBook}
Craig Huneke.
\newblock {\em Tight closure and its applications}, volume~88 of {\em CBMS
  Regional Conference Series in Mathematics}.
\newblock Published for the Conference Board of the Mathematical Sciences,
  Washington, DC; by the American Mathematical Society, Providence, RI, 1996.
\newblock With an appendix by Melvin Hochster.

\bibitem[Hun13]{HunekeHK}
Craig Huneke.
\newblock Hilbert-{K}unz multiplicity and the {F}-signature.
\newblock In {\em Commutative algebra}, pages 485--525. Springer, New York,
  2013.

\bibitem[HW02]{HaraWatanabe}
Nobuo Hara and Kei-Ichi Watanabe.
\newblock F-regular and {F}-pure rings vs. log terminal and log canonical
  singularities.
\newblock {\em J. Algebraic Geom.}, 11(2):363--392, 2002.

\bibitem[HY03]{HaraYoshida}
Nobuo Hara and Ken-Ichi Yoshida.
\newblock A generalization of tight closure and multiplier ideals.
\newblock {\em Trans. Amer. Math. Soc.}, 355(8):3143--3174, 2003.

\bibitem[ILL{\etalchar{+}}07]{24hours}
Srikanth~B. Iyengar, Graham~J. Leuschke, Anton Leykin, Claudia Miller, Ezra
  Miller, Anurag~K. Singh, and Uli Walther.
\newblock {\em Twenty-four hours of local cohomology}, volume~87 of {\em
  Graduate Studies in Mathematics}.
\newblock American Mathematical Society, Providence, RI, 2007.

\bibitem[KM98]{KollarMori}
J\'{a}nos Koll\'{a}r and Shigefumi Mori.
\newblock {\em Birational geometry of algebraic varieties}, volume 134 of {\em
  Cambridge Tracts in Mathematics}.
\newblock Cambridge University Press, Cambridge, 1998.
\newblock With the collaboration of C. H. Clemens and A. Corti, Translated from
  the 1998 Japanese original.

\bibitem[KR86]{KatzRatliff}
Daniel Katz and Louis~J. Ratliff, Jr.
\newblock On the symbolic {R}ees ring of a primary ideal.
\newblock {\em Comm. Algebra}, 14(5):959--970, 1986.

\bibitem[Kun76]{Kunz2}
Ernst Kunz.
\newblock On {N}oetherian rings of characteristic {$p$}.
\newblock {\em Amer. J. Math.}, 98(4):999--1013, 1976.

\bibitem[Lip69]{LipmanRationalSingularities}
Joseph Lipman.
\newblock Rational singularities, with applications to algebraic surfaces and
  unique factorization.
\newblock {\em Inst. Hautes \'{E}tudes Sci. Publ. Math.}, (36):195--279, 1969.

\bibitem[LS99]{LyubeznikSmith}
Gennady Lyubeznik and Karen~E. Smith.
\newblock Strong and weak {$F$}-regularity are equivalent for graded rings.
\newblock {\em Amer. J. Math.}, 121(6):1279--1290, 1999.

\bibitem[Mac96]{Maccrimmon}
Brian~Cameron Maccrimmon.
\newblock {\em Strong {F}-regularity and boundedness questions in tight
  closure}.
\newblock ProQuest LLC, Ann Arbor, MI, 1996.
\newblock Thesis (Ph.D.)--University of Michigan.

\bibitem[Mon83]{Monsky}
P.~Monsky.
\newblock The {H}ilbert-{K}unz function.
\newblock {\em Math. Ann.}, 263(1):43--49, 1983.

\bibitem[MPST19]{MaPolstraSchwedeTucker}
Linquan Ma, Thomas Polstra, Karl Schwede, and Kevin Tucker.
\newblock {$F$}-signature under birational morphisms.
\newblock {\em Forum Math. Sigma}, 7:e11, 20, 2019.

\bibitem[PT18]{PolstraTucker}
Thomas Polstra and Kevin Tucker.
\newblock {$F$}-signature and {H}ilbert-{K}unz multiplicity: a combined
  approach and comparison.
\newblock {\em Algebra Number Theory}, 12(1):61--97, 2018.

\bibitem[Ree58]{Rees58}
D.~Rees.
\newblock On a problem of {Z}ariski.
\newblock {\em Illinois J. Math.}, 2:145--149, 1958.

\bibitem[Rob85]{Roberts}
Paul~C. Roberts.
\newblock A prime ideal in a polynomial ring whose symbolic blow-up is not
  {N}oetherian.
\newblock {\em Proc. Amer. Math. Soc.}, 94(4):589--592, 1985.

\bibitem[Sch86]{Schenzel}
Peter Schenzel.
\newblock Finiteness of relative {R}ees rings and asymptotic prime divisors.
\newblock {\em Math. Nachr.}, 129:123--148, 1986.

\bibitem[Smi93]{SmithThesis}
Karen~Ellen Smith.
\newblock {\em Tight closure of parameter ideals and {F}-rationality}.
\newblock ProQuest LLC, Ann Arbor, MI, 1993.
\newblock Thesis (Ph.D.)--University of Michigan.

\bibitem[Smi97]{SmithRationalSingularities}
Karen~E. Smith.
\newblock {$F$}-rational rings have rational singularities.
\newblock {\em Amer. J. Math.}, 119(1):159--180, 1997.

\bibitem[SS10]{SchwedeSmith}
Karl Schwede and Karen~E. Smith.
\newblock Globally {$F$}-regular and log {F}ano varieties.
\newblock {\em Adv. Math.}, 224(3):863--894, 2010.

\bibitem[Tak04]{Takagi}
Shunsuke Takagi.
\newblock An interpretation of multiplier ideals via tight closure.
\newblock {\em J. Algebraic Geom.}, 13(2):393--415, 2004.

\bibitem[Tru98]{Trung}
Ng\^{o}~Vi\^{e}t Trung.
\newblock The {C}astelnuovo regularity of the {R}ees algebra and the associated
  graded ring.
\newblock {\em Trans. Amer. Math. Soc.}, 350(7):2813--2832, 1998.

\bibitem[Tuc12]{Tucker}
Kevin Tucker.
\newblock {$F$}-signature exists.
\newblock {\em Invent. Math.}, 190(3):743--765, 2012.

\bibitem[Wat94]{WatanabeCyclicCover}
Keiichi Watanabe.
\newblock Infinite cyclic covers of strongly {$F$}-regular rings.
\newblock In {\em Commutative algebra: syzygies, multiplicities, and birational
  algebra ({S}outh {H}adley, {MA}, 1992)}, volume 159 of {\em Contemp. Math.},
  pages 423--432. Amer. Math. Soc., Providence, RI, 1994.

\bibitem[Wil95]{Williams}
Lori~J. Williams.
\newblock Uniform stability of kernels of {K}oszul cohomology indexed by the
  {F}robenius endomorphism.
\newblock {\em J. Algebra}, 172(3):721--743, 1995.

\bibitem[WY04]{WY}
Kei-ichi Watanabe and Ken-ichi Yoshida.
\newblock Minimal relative {H}ilbert-{K}unz multiplicity.
\newblock {\em Illinois J. Math.}, 48(1):273--294, 2004.

\bibitem[Yao05]{Yaofinite}
Yongwei Yao.
\newblock Modules with finite {$F$}-representation type.
\newblock {\em J. London Math. Soc. (2)}, 72(1):53--72, 2005.

\end{thebibliography}

\end{document}